\documentclass[letterpaper]{article} 
\usepackage{times}  
\usepackage{helvet}  
\usepackage{courier}  
\usepackage[hyphens]{url}  
\usepackage{graphicx} 
\usepackage{natbib}  
\usepackage{caption} 
\usepackage[pagebackref]{hyperref}
\usepackage{times}

\renewcommand*{\backrefalt}[4]{%
    \ifcase #1 \footnotesize{(Not cited.)}%
    \or        \footnotesize{(Cited on page~#2)}%
    \else      \footnotesize{(Cited on pages~#2)}%
    \fi}

\usepackage{amssymb,amsmath,amscd,amsfonts,amstext,amsthm,bbm, enumerate, dsfont, mathtools}
\usepackage{array}
\usepackage{multicol}
\usepackage{algorithm}
\usepackage{varwidth}
\usepackage{import}
\usepackage{subcaption}
\usepackage{algorithmic}
\usepackage[none]{hyphenat}
\usepackage{footnote}
\usepackage{cleveref}
\usepackage{makecell}
\usepackage{threeparttable}
\usepackage{booktabs}
\usepackage{tcolorbox}

\usepackage{scrextend}

\graphicspath{ {images/} }

\newtheorem{theorem}{Theorem}
\newtheorem{proposition}{Proposition}
\newtheorem{lemma}{Lemma}

\newtheorem{assumption}{Assumption}
\newtheorem*{remark}{Remark}
\newtheorem{definition}{Definition}

\renewcommand{\algorithmicrequire}{\textbf{Input:}}
\renewcommand{\algorithmicensure}{\textbf{Output:}}

\DeclareMathOperator*{\argmin}{argmin}

\newcommand{\sumin}{\sum_{i=1}^n}

\newcommand{\avein}{\frac{1}{n}\sum_{i=1}^n}

\newcommand{\ec}[2][]{\ensuremath{\mathbb{E}_{#1} \left[#2\right]}}
\newcommand{\R}{\mathbb{R}}

\newcommand{\eqdef}{\stackrel{\text{def}}{=}}
\newcommand{\cO}{{\cal O}}

\newcommand{\meta}{F}
\newcommand{\opt}{x^*}
\newcommand {\aveis}[2]{\frac 1{|#2|} \sum_{#1 \in #2}}

\usepackage{tablefootnote}
\usepackage{changepage}
\usepackage{stackengine}
\usepackage{pifont}

\definecolor{mydarkgreen}{RGB}{39,130,67}

\definecolor{mydarkred}{RGB}{192,47,25}
\newcommand{\red}{\color{mydarkred}}

\newcommand{\cmark}{{\color{mydarkgreen}\ding{51}}}%
\newcommand{\xmark}{{\color{mydarkred} \ding{55}}}%
\newcommand{\outers}{\beta}
\newcommand{\inners}{\alpha}
\newcommand{\norm}[1]{\left \| #1 \right\|}
\newcommand{\norms}[1]{\left \| #1 \right\|^2}
\newcommand{\E}[1]{\mathbb{E}\left[#1\right] }

\newcommand{\Lmeta}{L_{\meta}}
\newcommand{\varopt}{\sigma_*^2}
\newcommand{\locstep}{\gamma}

\usepackage[colorinlistoftodos,bordercolor=orange,backgroundcolor=orange!20,linecolor=orange,textsize=scriptsize]{todonotes}

\usepackage{mdframed}
\mdfsetup{%
	linecolor=white,
	backgroundcolor=lightgray!50,
}

\usepackage[verbose=true,letterpaper]{geometry}
\AtBeginDocument{
	\newgeometry{
		textheight=9in,
		textwidth=6.4in,
		top=1in,
		headheight=12pt,
		headsep=25pt,
		footskip=30pt
	}
}

\newcommand{\footremember}[2]{%
	\footnote{#2}
	\newcounter{#1}
	\setcounter{#1}{\value{footnote}}%
}
\newcommand{\footrecall}[1]{%
	\footnotemark[\value{#1}]%
}

\def\<#1,#2>{\left\langle #1,#2\right\rangle}

\title{Convergence of First-Order Algorithms for Meta-Learning with Moreau Envelopes}

\author{
	Konstantin Mishchenko\footremember{Samsung} {Samsung AI Center, Cambridge, UK}
	\and Slavom\'ir Hanzely\footremember{KAUST} {King Abdullah University of Science and Technology, Saudi Arabia}	
	\and Peter Richt\'arik\footrecall{KAUST}
}
\begin{document}
\maketitle

\begin{abstract}
	In this work, we consider the problem of minimizing the sum of Moreau envelopes of given functions, which has previously appeared in the context of meta-learning and personalized federated learning. In contrast to the existing theory that requires running subsolvers until a certain precision is reached, we only assume that a finite number of gradient steps is taken at each iteration. As a special case, our theory allows us to show the convergence of First-Order Model-Agnostic Meta-Learning (FO-MAML) to the vicinity of a solution of Moreau objective. We also study a more general family of first-order algorithms that can be viewed as a generalization of FO-MAML. Our main theoretical achievement is a theoretical improvement upon the inexact SGD framework. In particular, our perturbed-iterate analysis allows for tighter guarantees that improve the dependency on the problem's conditioning. In contrast to the related work on meta-learning, ours does not require any assumptions on the Hessian smoothness, and can leverage smoothness and convexity of the reformulation based on Moreau envelopes. Furthermore, to fill the gaps in the comparison of FO-MAML to the Implicit MAML (iMAML), we show that the objective of iMAML is neither smooth nor convex, implying that it has no convergence guarantees based on the existing theory.
\end{abstract}

\section{Introduction}
Efficient optimization methods for empirical risk minimization have helped the breakthroughs in many areas of machine learning such as computer vision~\cite{krizhevsky2012imagenet} and speech recognition~\cite{hinton2012deep}. More recently, elaborate training algorithms have enabled fast progress in the area of meta-learning, also known as learning to learn~\cite{schmidhuber1987evolutionary}. At its core lies the idea that one can find a model capable of retraining for a new task with just a few data samples from the task. Algorithmically, this corresponds to solving a bilevel optimization problem~\cite{franceschi2018bilevel}, where the inner problem corresponds to a single task, and the outer problem is that of minimizing the post-training error on a wide range of tasks.

The success of Model-Agnostic Meta-Learning (MAML) and its first-order version (FO-MAML) \cite{finn2017model} in meta-learning applications has propelled the development of new gradient-based meta-learning methods. However, most new algorithms effectively lead to new formulations of meta-learning. For instance, iMAML \cite{rajeswaran2019meta} and proximal meta-learning \cite{zhou2019efficient} define two MAML-like objectives with implicit gradients, while Reptile \cite{nichol2018first} was proposed without defining any objective at all. These dissimilarities cause fragmentation of the field and make it particularly hard to have a clear comparison of meta-learning theory. Nonetheless, having a good theory helps to compare algorithms as well as identify and fix their limitations.

Unfortunately, for most of the existing methods, the theory is either incomplete as is the case with iMAML or even completely missing. In this work, we set out to at least partially mitigate this issue by proposing a new analysis for minimization of Moreau envelopes. We show that a general family of algorithms with multiple gradient steps is stable on this objective and, as a special case, we obtain results even for FO-MAML. Previously, FO-MAML was viewed as a heuristic to approximate MAML \cite{fallah2020convergence}, but our approach reveals that FO-MAML can be regarded as an algorithm for a the sum of Moreau envelopes. While both perspectives show only approximate convergence, the main justification for the sum of Moreau envelopes is that requires unprecedentedly mild assumptions. In addition, the Moreau formulation of meta-learning does not require Hessian information and is easily implementable by any first-order optimizer, which \citet{zhou2019efficient} showed to give good empirical performance.

\subsection{Related work}\label{sec:related_work}

\begin{table*}[t]
    \centering
\setlength\tabcolsep{3.5pt} 
  \begin{threeparttable}[b]
    {\scriptsize
	\renewcommand\arraystretch{1.8}
	\caption{A summary of related work and conceptual differences to our approach. We mark as ``N/A'' unknown properties that have not been established in prior literature or our work. We say that $F_i$ ``Preserves convexity'' if for convex $f_i$, $F_i$ is convex as well, which implies that $F_i$ has no extra local minima or saddle points. We say that $F_i$ ``Preserves smoothness'' if its gradients are Lipschitz whenever the gradients of $f_i$ are, which corresponds to more stable gradients. We refer to \cite{fallah2020convergence} for the claims regarding nonconvexity and nonsmoothness of the MAML objective.}
  \label{tab:compare_with_others}
	\centering 
	\begin{tabular}{ccccccccc}\toprule[.1em]
		Algorithm & \makecell{$\meta_i$: meta-loss of task $i$} & \makecell{Hessian-\\ free} & \makecell{Arbitrary\\ number of steps} & \makecell{No matrix\\ inversion} & \makecell{Preserves\\ convexity} & \makecell{Preserves\\ smoothness} & Reference \\
		\midrule
		 MAML & $f_i(x-\alpha \nabla f_i(x))$ & \xmark & \xmark & \cmark & \xmark & \xmark &  \cite{finn2017model} \\
		 \makecell{Multi-step\\ MAML} & $f_i(GD(f_i, x))$\tnote{\color{red}(1)} & \xmark & \cmark & \cmark & \xmark & \xmark &  \makecell{\cite{finn2017model}\\ \cite{ji2020theoretical}} \\
		 iMAML\tnote{\red(2)} & \makecell{$f_i(z_i(x))$, where\\ $z_i(x)=x-\alpha \nabla f_i(z_i(x))$}& \xmark & \cmark & \xmark & \makecell{\xmark\\{\scriptsize (Theorem~\ref{th:imaml_nonconvex})}} & \makecell{\xmark\\{\scriptsize(Theorem~\ref{th:imaml_nonsmooth})}} & \cite{rajeswaran2019meta}  \\
		  Reptile & N/A\tnote{\red(3)}  & \cmark & \cmark & \cmark & N/A & N/A & \cite{nichol2018first} \\
		  \makecell{FO-MAML\\ (original)} & $f_i(x-\alpha \nabla f_i(x))$ & \cmark & \xmark & \cmark & \xmark & \xmark & \cite{finn2017model} \\
		  \makecell{Meta-MinibatchProx} & $\min\limits_{x_i}\{f_i(x_i) + \frac{1}{2\alpha}\|x_i-x\|^2 \}$ & \cmark & \xmark\tnote{\red(4)} & \cmark & \cmark & \cmark &   \makecell{\cite{zhou2019efficient}} \\
		  \midrule
		  \makecell{\textbf{FO-MuML}\\ \textbf{(extended FO-MAML)}} & $\min\limits_{x_i}\{f_i(x_i) + \frac{1}{2\alpha}\|x_i-x\|^2 \}$ & \cmark & \cmark & \cmark & \cmark & \cmark &   \makecell{\textbf{This work}} \\
		\bottomrule[.1em]
    \end{tabular}
    }
    \begin{tablenotes}
      {\footnotesize
        \item [\color{red}(1)] Multi-step MAML runs an inner loop with gradient descent applied to task loss $f_i$, so the objective of multi-step MAML is $F_i(x)=f_i(x_s(x))$, where $x_0=x$ and $x_{j+1}=x_j - \alpha \nabla f_i(x_j)$ for $j=0,\dotsc, s-1$.
        \item [\color{red}(2)] To the best of our knowledge, iMAML is not guaranteed to work; \citet{rajeswaran2019meta} studied only the approximation error for gradient computation, see the discussion in our special section on iMAML. 
        \item [\color{red}(3)] Reptile was proposed as an algorithm on its own, without providing any optimization problem. This makes it hard to say how it affects smoothness and convexity. \citet{balcan2019provable} and \citet{khodak2019adaptive} studied convergence of Reptile on the average loss over the produced iterates, i.e., $F_i(x)=\frac{1}{m}\sum_{j=0}^s f_i(x_j)$, where $x_0=x$ and $x_{j+1}=x_j - \alpha \nabla f_i(x_j)$ for $j=0,\dotsc, s-1$. Analogously to the loss of MAML, this objective seems nonconvex and nonsmooth.
        \item [\color{red}(4)] \citet{zhou2019efficient} assumed that the subproblems are solved to precision $\varepsilon$, i.e., $x_i$ is found such that $\|\nabla f_i(x_i) + \frac{1}{\alpha}(x_i - x)\|\le \varepsilon$ with an absolute constant $\varepsilon$.
      }
    \end{tablenotes}
  \end{threeparttable}
\end{table*}
MAML \cite{finn2017model} has attracted a lot of attention due to its success in practice. Many improvements have been proposed for MAML, for instance, \cite{zhou2020task} suggested augmenting each group of tasks with its own global variable, and \cite{antoniou2018train} proposed MAML++ that uses intermediate task losses with weights to improve the stability of MAML. \cite{rajeswaran2019meta} proposed iMAML that makes the objective optimizer-independent by relying on \emph{implicit} gradients. \citet{zhou2019efficient} used a similar implicit objective to that of iMAML with an additional regularization term that, unlike iMAML, does not require inverting matrices. Reptile \cite{nichol2018first} is an even simpler method that merely runs gradient descent on each sampled task. Based on generalization guarantees, \cite{zhou2020task} also provided a trade-off between the optimization and statistical errors for a multi-step variant MAML, which shows that it may not improve significantly from increasing the number of gradient steps in the inner loop. We refer to \cite{hospedales2021meta} for a recent survey of the literature on meta-learning with neural networks.

On the theoretical side, the most relevant works to ours is that of \cite{zhou2019efficient}, whose main limitation is that it requires a high-precision solution of the inner problem in Moreau envelope at each iteration. Another relevant work that studied convergence of MAML and FO-MAML on the standard MAML objective is by \cite{fallah2020convergence}, but they do not provide any guarantees for the sum of Moreau envelopes and their assumptions are more stringent. \citet{fallah2020convergence} also study a Hessian-free variant of MAML, but its convergence guarantees still require posing assumptions on the Hessian Lipschitzness and variance.

Some works treat meta-learning as a special case of compositional optimization~\cite{sun2021optimal} or bilevel programming \cite{franceschi2018bilevel} and develop theory for the more general problem. Unfortunately, both approaches lead to worse dependence on the conditioning numbers of both inner and outer objective, and provide very pessimistic guarantees. Bilevel programming, even more importantly, requires computation of certain inverse matrices, which is prohibitive in large dimensions. One could also view minimization-based formulations of meta-learning as instances of empirical risk minimization, for which FO-MAML can be seen as instance of inexact (biased) SGD. For example, \cite{ajalloeian2020analysis} analyzed SGD with deterministic bias and some of our proofs are inspired by theirs, except in our problem the bias is not deterministic. We will discuss the limitations of their approach in the section on inexact SGD.

Several works have also addressed meta-learning from the statistical perspective, for instance, \citet{yoon2018bayesian} proposed a Bayesian variant of MAML, and \citet{finn2019online} analyzed convergence of MAML in online learning. Another example is the work of \citet{konobeev2021distributiondependent} who studied the setting of linear regression with task-dependent solutions that are sampled from same normal distribution. These directions are orthogonal to ours, as we want to study the optimization properties of meta-learning.

\section{Background and mathematical formulation}

Before we introduce the considered formulation of meta-learning, let us provide the problem background and define all notions. As the notation in meta-learning varies between papers, we correspond our notation to that of other works in the next subsection.
\subsection{Notation}
We assume that training is performed over $n$ tasks with task losses $f_1,\dotsc, f_n$ and we will introduce \emph{implicit} and \emph{proximal} meta-losses $\{F_i\}$ in the next section. 
We denote by $x$ the vector of parameters that we aim to train, which is often called \emph{model}, \emph{meta-model} or \emph{meta-parameters} in the meta-learning literature, and \emph{outer variable} in the bilevel literature. Similarly, given task $i$, we denote by $z_i$ the \emph{task-specific parameters} that are also called as \emph{ground model}, \emph{base-model}, or \emph{inner variable}. We will use letters $\alpha, \beta, \gamma$ to denote scalar hyper-parameters such as stepsize or regularization coefficient.

Given a function $\varphi(\cdot)$, we call the following function its \emph{Moreau envelope}:
\begin{align*}
	\Phi(x) = \min_{z\in\R^d}\left\{\varphi(x) + \frac{1}{2\alpha}\|z-x\|^2 \right\},
\end{align*}
where $\alpha>0$ is some parameter. Given the Moreau envelope $F_i$ of a task loss $f_i$, we denote by $z_i(x)$ the solution to the inner objective of $F_i$, i.e., $z_i(x)\eqdef \argmin_{z\in\R^d} \left\{f_i(z) + \frac{1}{2\alpha}\|z - x\|^2\right\}$.

Finally, let us introduce some standard function properties that are commonly used in the optimization literature \cite{Nesterov2013}.
	\begin{definition}
		We say that a function $\varphi(\cdot)$ is $L$-\emph{smooth} if its gradient is $L$-Lipschitz, i.e., for any $x, y\in\R^d$,
		\begin{align*}
			\|\nabla \varphi(x) - \nabla \varphi(y)\|\le L\|x-y\|.
		\end{align*}
	\end{definition}
	\begin{definition}
		Given a function $\varphi(\cdot)$, we call it $\mu$-\emph{strongly convex} if it satisfies for any $x, y\in\R^d$,
		\begin{align*}
			\varphi(y)\ge \varphi(x) + \<\nabla \varphi(x), y-x> + \frac{\mu}{2}\|y-x\|^2.
		\end{align*}
		If the property above holds with $\mu=0$, we call $\varphi$ to be \emph{convex}. If the property does not hold even with $\mu=0$, we say that $\varphi$ is \emph{nonconvex}.
	\end{definition}

\subsection{MAML objective}
Assume that we are given $n$ tasks, and that the performance on task $i$ is evaluated according to some loss function $f_i(x)$. MAML has been proposed as an algorithm for solving the following objective:
\begin{align}
\min_{x\in\R^d} \frac{1}{n}\sumin f_i(x - \alpha \nabla f_i(x)), \label{eq:bad_problem}
\end{align}
where $\alpha>0$ is a stepsize. Ignoring for simplicity minibatching, MAML update computes the gradient of a task meta-loss $\varphi_i(x)=f_i(x - \alpha \nabla f_i(x))$ through backpropagation and can be explicitly written as
\begin{align*}
	x^{k+1}
	&= x^k - \beta \left(\mathbf{I}-\alpha\nabla^2 f_i(x^k) \right)\nabla f_i (x^k-\alpha\nabla f_i(x^k)),     \tag{MAML update}
\end{align*}
where $\beta>0$ is a stepsize, $i$ is sampled uniformly from $\{1,\dotsc, n\}$ and $\mathbf{I}\in\R^{d\times d}$ is the identity matrix. Sometimes, MAML update evaluates the gradient of $\varphi_i$ using an additional data sample, but \citet{bai2021important} recently showed that this is often unnecessary, and we, thus, skip it.

Unfortunately, objective~\eqref{eq:bad_problem} might be nonsmooth and nonconvex even if the task losses $\{f_i\}$ are convex and smooth~\cite{fallah2020convergence}. Moreover, if we generalize this objective for more than one gradient step inside $f_i(\cdot)$, its smoothness properties deteriorate further, which complicates the development and analysis of multistep methods. 

\subsection{iMAML objective}\label{sec:imaml}
To avoid differentiating through a graph, \citet{rajeswaran2019meta} proposed an alternative objective to~\eqref{eq:bad_problem} that replaces the gradient step inside each function with an \emph{implicit} gradient step. In particular, if we define $z_i(x)\eqdef \argmin_{z\in\R^d} \left\{f_i(z) + \frac{1}{2\alpha}\|z - x\|^2\right\}$, then the objective of iMAML is
\[
	\min_{x\in\R^d} \avein f_i\left(x-\alpha \nabla f_i(z_i(x))\right).
\]
The idea of iMAML is to optimize this objective during training so that at inference, given a new function $f_{n+1}$ and solution $x_{\mathrm{iMAML}}$ of the problem above, one can find an approximate solution to $\min_{z\in\R^d} \left\{f_{n+1}(z) + \frac{1}{2\alpha}\|z - x_{\mathrm{iMAML}}\|^2\right\}$ and use it as a new model for task $f_{n+1}$.

\citet{rajeswaran2019meta} proved, under some mild assumptions, that one can efficiently obtain an estimate of the gradient of $\varphi_i(x)\eqdef f_i\left(x-\alpha \nabla f_i(z_i(x))\right)$ with access only to gradients and Hessian-vector products of $f_i$, which rely on standard backpropagation operations. In particular, \citet{rajeswaran2019meta} showed that
\[
	\nabla \varphi_i(x)
	= \left(\mathbf{I} + \alpha \nabla^2 f_i(z(x)) \right)^{-1} \nabla f_i(z(x)),
\]
where $\mathbf{I}$ is the identity matrix, and they proposed to run the conjugate gradient method to find $\nabla \varphi_i(x)$.
However, it is not shown in \cite{rajeswaran2019meta} if the objective of iMAML is solvable and what properties it has. Moreover, we are not aware of any result that would show when the problem is convex or smooth. Since SGD is not guaranteed to work unless the objective satisfies at least some properties \cite{zhang2020complexity}, nothing is known about convergence of SGD when applied to the iMAML objective. 

As a sign that the problem is rather ill-designed, we present the following theorem that gives a negative example on the problem's convexity.
\begin{theorem}\label{th:imaml_nonconvex}
	There exists a convex function $f$ with Lipschitz gradient and Lipschitz Hessian such that the iMAML meta-objective $\varphi(x)\eqdef f(z(x))$ is nonconvex, where $z(x)=x - \alpha \nabla f(z(x))$.
\end{theorem}
Similarly, we also show that the objective of iMAML may be harder to solve due to its worse smoothness properties as given by the next theorem.
\begin{theorem}\label{th:imaml_nonsmooth}
	There exists a convex function $f$ with Lipschitz gradient and Lipschitz Hessian such that the iMAML meta-objective $\varphi(x)\eqdef f(z(x))$ is nonsmooth for any $\alpha>0$, where $z(x)=x - \alpha \nabla f(z(x))$.
\end{theorem}


\begin{algorithm}[t]
   \caption{FO-MAML: First-Order MAML}
   \label{alg:fo_maml}
\begin{algorithmic}[1]
   \STATE \textbf{Input:} $x^0$, $\alpha, \beta > 0$\\
   \FOR{$k=0,1,\dotsc$}
        \STATE Sample a subset of tasks $T_k$
        \FOR{each sampled task $i$ \textbf{in} $T_k$}
            \STATE $z_i^k = x^k - \alpha \nabla f_{i}(x^k)$
        \ENDFOR
        \STATE $x^{k+1} = x^k - \beta \frac{1}{|T_k|}\sum_{i\in T_k} \nabla f_i(z_i^k)$
 \ENDFOR
\end{algorithmic}
\end{algorithm}

\subsection{Our main objective: Moreau envelopes}\label{sec:our_reformulation}
In this work we consider the following formulation of meta-learning
\begin{align}
&\min_{x\in\R^d}  \meta(x) \eqdef \frac{1}{n}\sum_{i=1}^n \meta_i(x),\quad \label{eq:new_pb} \\
&\text{where}\quad
\meta_i(x)\eqdef \min_{z\in\R^d} \left\{f_i(z) + \frac{1}{2\alpha}\|z - x\|^2\right\}, \notag
\end{align}
and $\alpha>0$ is a parameter controlling the level of adaptation to the problem. In other words, we seek to find a parameter vector $x$ such that somewhere close to $x$ there exists a vector $z_i$ that verifies that $f_i(z)$ is sufficiently small. This formulation of meta-learning was first introduced by \citet{zhou2019efficient} and it has been used by \citet{Hanzely2020} and \citet{t2020personalized} to study personalization in federated learning.

Throughout the paper we use the following variables for minimizers of meta-problems $\meta_i$:
\begin{align}
	z_i(x) \eqdef \argmin_{z\in\R^d} \left\{f_i(z) + \frac{1}{2\alpha}\|z - x\|^2\right\},  i=1,\dotsc, n. \label{eq:z_i}
\end{align}
One can notice that if $\alpha\to 0$, then $\meta_i(x)\approx f_i(x)$, and Problem~\eqref{eq:new_pb} reduces to the well-known empirical risk minimization:
\begin{align*}
\min_{x\in\R^d} f(x)
\eqdef \frac{1}{n}\sum_{i=1}^n f_i(x).
\end{align*}
If, on the other hand, $\alpha\to +\infty$, the minimization problem in~\eqref{eq:new_pb} becomes essentially independent of $x$ and it holds $z_i(x)\approx \argmin_{z\in\R^d} f_i(z)$. Thus, one has to treat the parameter $\alpha$ as part of the objective that controls the similarity between the task-specific parameters.

We denote the solution to Problem~\eqref{eq:new_pb} as
\begin{align}
x^* \eqdef \arg\min_{x\in\R^d} \meta(x).
\end{align}
One can notice that $F(x)$ and $x^*$ depend on $\alpha$. For notational simplicity, we keep $\alpha$ constant throughout the paper and do not explicitly write the dependence of $x^*, F, F_1, z_1, \dotsc, F_n, z_n$ on $\alpha$.

\subsection{Formulation properties}\label{sec:our_formulation}
We will also use the following quantity to express the difficulty of Problem~\eqref{eq:new_pb}:
\begin{align}
\sigma_*^2\eqdef \frac{1}{n}\sumin \|\nabla \meta_i(\opt)\|^2.  \label{eq:def_sigma}
\end{align}
Because $\nabla \meta(\opt)=0$ by first-order optimality of $\opt$, $\sigma_*^2$ serves as a measure of gradient variance at the optimum. Note that $\sigma_*$ is always finite because it is defined on a single point, in contrast to the \emph{maximum} gradient variance over all space, which might be infinite.

Now let's discuss properties of our formulation \ref{eq:new_pb}. Firstly, we state a standard result from~\cite{beck-book-first-order}.

\begin{proposition}[Theorem 6.60 in~\cite{beck-book-first-order}]\label{pr:moreau_is_smooth}
	Let $\meta_i$ be defined as in \cref{eq:new_pb} and $z_i(x)$ be defined as in \cref{eq:z_i}.
	If $f_i$ is convex, proper and closed, then $\meta_i$ is differentiable and $\frac{1}{\alpha}$-smooth:
	\begin{align}
	&\nabla \meta_i(x)
	= \frac{1}{\alpha}(x-z_i(x)) = \nabla f_i(z_i(x)), \label{eq:implicit} \\
	&\|\nabla \meta_i(x) - \nabla \meta_i(y)\|\le \frac{1}{\alpha}\|x-y\|.
	\end{align}
\end{proposition}

The results above only hold for convex functions, while in meta-learning, the tasks are often defined by training a neural network, whose landscape is nonconvex. To address such applications, we also refine \Cref{pr:moreau_is_smooth} in the lemma bellow, which also improves the smoothness constant in the convex case. This result is similar to Lemma 2.5 of \cite{davis2021proximal}, except their guarantee is a bit weaker because they consider more general assumptions.

\begin{lemma}\label{lem:moreau_is_str_cvx_and_smooth}
	Let function $f_i$ be $L$-smooth.
	\begin{itemize}
		\item If $f_i$ is nonconvex and $\alpha<\frac{1}{L}$, then $\meta_i$ is $\frac{L}{1-\alpha L}$-smooth. If $\alpha \le \frac{1}{2L}$, then $\meta_i$ is $2L$-smooth.
		\item If $f_i$ is convex, then $\meta_i$ is $\frac{L}{1+\alpha L}$-smooth. Moreover, for any $\alpha$, it is $L$-smooth.
		\item If $f_i$ is $\mu$-strongly convex, then $\meta_i$ is $\frac{\mu}{1+\alpha\mu}$-strongly convex. If $\alpha \le \frac{1}{\mu}$, then $\meta_i$ is $\frac{\mu}{2}$-strongly convex.
	\end{itemize}
	Whenever $\meta_i$ is smooth, its gradient is given as in equation~\eqref{eq:implicit}, i.e., $\nabla \meta_i(x) = \nabla f_i(z_i(x))$.
\end{lemma}
The takeaway message of \Cref{lem:moreau_is_str_cvx_and_smooth} is that the optimization properties of $\meta_i$ are always at least as good as those of $f_i$ (up to constant factors). Furthermore, the \emph{conditioning}, i.e., the ratio of smoothness to strong convexity, of $\meta_i$ is upper bounded, up to a constant factor, by that of $f_i$. And even if $f_i$ is convex but nonsmooth ($L\to+\infty)$, $F_i$ is still smooth with constant $\frac{1}{\alpha}$.

Finally, note that computing the exact gradient of $F_i$ requires solving its inner problem as per equation~\eqref{eq:implicit}. Even if the gradient of task $\nabla f_i(x)$ is easy to compute, we still cannot obtain $ \nabla \meta_i(x)$ through standard differentiation or backpropagation. However, one can approximate $\nabla \meta_i(x)$ in various ways, as we will discuss later.

\section{Can we analyze FO-MAML as inexact SGD?}\label{sec:inexact_sgd}
As we mentioned before
, the prior literature has viewed FO-MAML as an inexact version of MAML for problem~\eqref{eq:bad_problem}. If, instead, we are interested in problem~\eqref{eq:new_pb}, one could still try to take the same perspective of inexact SGD and see what convergence guarantees it gives for~\eqref{eq:new_pb}. The goal of this section, thus, is to refine the existing theory of inexact SGD to make it applicable to FO-MAML. We will see, however, that such approach if fundamentally limited and we will present a better alternative analysis in a future section. 
\subsection{Why existing theory is not applicable}
Let us start with a simple lemma for FO-MAML that shows why it approximates SGD for objective~\eqref{eq:new_pb}.
\begin{lemma}\label{lem:approx_implicit}
	Let task losses $f_i$ be $L$--smooth and $\alpha>0$. Given $i$ and $x\in\R^d$, we define recursively $z_{i,0} \eqdef x$ and $z_{i,j+1} \eqdef {\color{blue}x} - \alpha \nabla f_i({\color{mydarkred}z_{i,j}})$. Then, it holds for any $s\ge 0$
	\begin{align*}
		\left\| \nabla f_i(z_{i,s}) - \nabla \meta_i(x) \right\| \le (\alpha L)^{s+1} \|\nabla \meta_i(x)\|.
	\end{align*}
	In particular, the iterates of FO-MAML (\Cref{alg:fo_maml}) satisfy for any $k$
	\begin{align*}
		\left\| \nabla f_i(z_i^k) - \nabla \meta_i(x^k) \right\| \le (\alpha L)^2 \|\nabla \meta_i(x^k)\|.
	\end{align*}
\end{lemma}
 \Cref{lem:approx_implicit} shows that FO-MAML approximates SGD step with error proportional to the stochastic gradient norm. Therefore, we can write
\begin{align*}
 	\nabla f_i(z_i^k)
 	= \nabla F(x^k) + \underbrace{\nabla F_i(x^k) - \nabla F(x^k)}_{\eqdef \xi_i^k\ (\mathrm{noise})} + \underbrace{b_i^k}_{\mathrm{bias}},
\end{align*}  
where it holds $\mathbb{E}[\xi_i^k]=0$, and $b_i^k$ is a bias vector that also depends on $i$ but does not have zero mean. The best known guarantees for inexact SGD are provided by \citet{ajalloeian2020analysis}, but they are, unfortunately, not applicable because their proofs use independence of $\xi_i^k$ and $b_i^k$. The analysis of \citet{zhou2019efficient} is not applicable either because their inexactness assumption requires the error to be smaller than a predefined constant $\varepsilon$, while the error in \Cref{lem:approx_implicit} can be unbounded. To resolve these issues, we provide a refined analysis in the next subsection.

\subsection{A new result for inexact SGD}
For strongly convex objectives, we give the following result by modifying the analysis of \citet{ajalloeian2020analysis}.

\begin{theorem}[Convergence of FO-MAML, weak result]\label{th:fo_maml}
	Let task losses $f_1,\dotsc, f_n$ be $L$-smooth and $\mu$-strongly convex. If $|T_k|=\tau$ for all $k$, $\outers \leq \frac 1 {20L}$ and $\alpha\le \frac{1}{4\sqrt{\kappa}L}$, where $\kappa\eqdef \frac{L}{\mu}$, then for the iterates $x^1, x^2\dots$ of \Cref{alg:fo_maml}, it holds
	\begin{align*}
		\ec{\|x^k-\opt\|^2}
		&\le \left(1 - \frac{\outers\mu}{4}\right)^k\|x^0-\opt\|^2 
		 + \frac{16}{\mu} \left( \frac {2\alpha^2 L^2} {\mu} + \frac \outers \tau + \outers  \right) \sigma_*^2.
	\end{align*}
\end{theorem}
Let us try to compare this result to that of vanilla SGD as studied by \citet{Gower2019}. Since the first term decreases exponentially, it requires us $\cO\left(\frac{1}{\beta\mu}\log\frac{1}{\varepsilon} \right)$ iterations to make it smaller than $\varepsilon$. The second term, on the other hand, only decreases if we decrease $\alpha$ and $\beta$. Decreasing $\beta$ corresponds to using decreasing stepsizes in SGD, which is fine, but $\alpha$ is a parameter that defines the objective, so in most cases, we do not want to decrease it. Moreover, the assumptions of \Cref{th:fo_maml} require $\alpha$ to be smaller than $\frac{1}{\sqrt{\kappa}L}$, which seems quite restrictive. This is the main limitation of this result as it shows that FO-MAML as given in \Cref{alg:fo_maml} may not converge to the problem solution.

\begin{algorithm}[t]
   \caption{FO-MuML: First-Order Multistep Meta-Learning (general formulation)}
   \label{alg:mamlP}
\begin{algorithmic}[1]
   \STATE \textbf{Input:}$x^0$, $\beta>0$, accuracy $\delta\geq0$ or $\varepsilon\ge 0$.
   \FOR{$k=0,1,\dotsc$}
        \STATE Sample a subset of tasks $T_k$
        \FOR{each sampled task $i$ \textbf{in} $T_k$}
            \STATE  Find $z_i^k$ s.t.\ $\norm{\frac 1 \inners \left(x^k -z_i^k \right) - \nabla \meta_i(x^k)} \leq \delta \norm{\nabla \meta_i(x^k)}$          
        \ENDFOR
        \STATE $x^{k+1} = x^k - \beta\frac{1}{|T_k|}\sum_{i\in T_k} \nabla f_i(z_i^k)$
   \ENDFOR
\end{algorithmic}
\end{algorithm}
To fix the nonconvergence of FO-MAML, let us turn our attention to \Cref{alg:mamlP}, which may perform multiple first-order steps.

\begin{theorem} \label{th:convergence_of_mamlP}
	Let task losses $f_1,\dotsc, f_n$ be $L$-smooth and $\mu$-strongly convex. If $|T_k|=\tau$ for all $k$, $\alpha\le \frac{1}{L}, \outers \leq \frac 1 {20L}$, and $\delta \leq \frac 1 {4 \sqrt{ \kappa}}$, where $\kappa\eqdef \frac{L}{\mu}$, then the iterates of \Cref{alg:mamlP} satisfy
	\begin{align*}
		\ec{\|x^k-\opt\|^2}
		&\le \left(1 - \frac{\outers\mu}{4}\right)^k\|x^0-\opt\|^2
		 + \frac{16}{\mu} \left( \frac {2\delta^2} {\mu} + \frac \outers \tau + \outers \delta^2 \right) \sigma_*^2.
	\end{align*}
\end{theorem}
The result of \Cref{th:convergence_of_mamlP} is better than that of \Cref{th:fo_maml} since it only requires the inexactness parameter $\delta$ to go to 0 rather than $\alpha$, so we can solve the meta-learning problem \eqref{eq:new_pb} for any $\alpha\le \frac{1}{L}$. The rate itself, however, is not optimal, as we show in the next section with a more elaborate approach.

\section{Improved theory}\label{sec:better_theory}
\begin{algorithm}[t]
   \caption{FO-MuML (example of implementation)}
   \label{alg:maml2}
\begin{algorithmic}[1]
   \STATE \textbf{Input:} $x^0$, number of steps $s$, $\alpha > 0$, $\beta>0$
   \FOR{$k=0,1,\dotsc$}
        \STATE Sample a subset of tasks $T_k$
        \FOR{each sampled task $i$ \textbf{in} $T_k$}
            \STATE $z_{i, 0}^k = x^k$
            \FOR{$l = 0, \dotsc, s-1$}
                \STATE $z^k_{i, l+1} = {\color{blue} x^k} - \alpha \nabla f_{i}({\color{red}z_{i, l}^k})$
            \ENDFOR
            \STATE $z_i^k = z_{i,s}^k$
        \ENDFOR
        \STATE $x^{k+1} = x^k - \beta\frac{1}{|T_k|}\sum_{i\in T_k} \nabla f_i(z_i^k)$
   \ENDFOR
\end{algorithmic}
\end{algorithm}
In this section, we provide improved convergence theory of FO-MAML and FO-MuML based on a sequence of virtual iterates that appear only in the analysis. Surprisingly, even though the sequence never appears in the algorithm, it allows us to obtain tighter convergence bounds.

\subsection{Perturbed iterate is better than inexact gradient}
Before we introduce the sequence, let us make some observations from prior literature on inexact and biased variants of SGD. For instance, the literature on asynchronous optimization has established that getting gradient at a wrong point does not significantly worsen its rate of convergence \cite{mania2017perturbed}. A similar analysis with additional virtual sequence was used in the so-called error-feedback for compression \cite{stich2018sparsified}, where the goal of the sequence is to follow the path of \emph{exact} gradients even if \emph{compressed} gradients are used by the algorithm itself. Motivated by these observations, we set out to find a virtual sequence that could help us analyze FO-MAML.
\subsection{On what vector do we evaluate the gradients?}
The main difficulty that we face is that we never get access to the gradients of $\{F_i\}$ and have to use the gradients of $\{f_i\}$. However, we would still like to write
\begin{align*}
x^{k+1} 
=x^k - \frac \inners \tau \sum_{i \in T_k} \nabla f_i(z_i^k) 
= x^k - \frac \inners \tau \sum_{i \in T_k} \nabla \meta_i(y_i^k)
\end{align*}
for some point $y_i^k$. If this is possible, using point $y_i^k$ would allow us to avoid working with functions $f_i$ in some of our recursion.


Why exactly would this sequence help? As mentioned before, FO-MAML is a biased method, so we cannot evaluate expectation of $\E{\nabla f_i(z_i^k)}$. However, if we had access to $\nabla F_i(x^k)$, its expectation would be exactly $\nabla F(x^k)$. This suggests that if we find $y_i^k$ that satisfies
$\nabla \meta_i(y_i^k) \approx \nabla \meta_i(x^k)$, then 
\[
	x^{k+1} = x^k - \frac \inners \tau \sum_{i \in T_k} \nabla \meta_i(y_i^k) \approx x^k - \frac \inners \tau \sum_{i \in T_k} \nabla \meta_i(x^k),
\]
which would allow us to put the bias \emph{inside} the gradient. 

Fortunately, objective \eqref{eq:new_pb} allows us to find such point easily. In particular, for Moreau Envelopes, the following proposition holds.

\begin{lemma}\label{lem:explicit_grad_to_implicit}
	For any points $z, y \in \R^d$ it holds $y= z + \inners \nabla f_i(z)$ if and only if $z = y - \inners \nabla \meta_i(y)$. Therefore, given $z$, we can define $y=z+\alpha \nabla f_i(z)$ and obtain $\nabla f_i(z)=\nabla F_i(y)$.
\end{lemma}
\begin{proof}
	The result follows immediately from the last statement of Lemma~\ref{lem:moreau_is_str_cvx_and_smooth}.
\end{proof}
The second part of \Cref{lem:explicit_grad_to_implicit} is exactly what we need. Indeed, we can choose $y_i^k \eqdef z_{i}^k + \inners \nabla  f_i(z_{i}^k)$ so that $z_{i}^k = y_i^k - \inners \nabla \meta_i(y_i^k)$ and $\nabla f_i(z_i^k)=\nabla F_i(y_i^k)$. As we have explained, this can help us to tackle the bias of FO-MAML.

\subsection{Main results}
We have established the existence of variables $y_i^k$ such that $\nabla f_i(z_i^k)=\nabla F_i(y_i^k)$. This allows us to write 
\begin{align*}
 	\nabla f_i(z_i^k)
 	= \nabla F_i(y_i^k) 
 	= \nabla F(x^k) + \underbrace{\nabla F_i(x^k) - \nabla F(x^k)}_{\mathrm{noise}}
 	+ \underbrace{\nabla F_i(y_i^k) - \nabla F_i(x^k)}_{\textrm{reduced bias}}&.
\end{align*}  
As the next theorem shows, we can use this to obtain convergence guarantee to a neighborhood even with a small number of steps in the inner loop.
\begin{theorem} \label{th:convengence_of_mamlP_no_stepsize}
	Consider the iterates of \Cref{alg:mamlP} (with general $\delta$) or \Cref{alg:fo_maml} (for which $\delta=\alpha L$).
Let task losses be $L$--smooth and $\mu$--strongly convex and let objective parameter satisfy $\inners \leq \frac {1}{\sqrt 6 L}$. Choose stepsize $ \beta \leq \frac \tau {4 L}$, where $\tau = |T_k|$ is the batch size. Then we have
	\begin{align*}
		\E{\norm{x^k-x^*}^2} &\leq \left(1 - \frac {\beta \mu}{12}  \right)^k \norm{x^0 - x^*}^2
	  + \frac { 6\left( \frac \beta \tau + 3 \delta^2 \inners^2 L \right) \varopt} {\mu}.
	\end{align*}
\end{theorem}
Similarly to \Cref{th:fo_maml}, the theorem above guarantees convergence to a neighborhood only. However, the radius of convergence is now $\cO\left(\frac{\frac{\beta}{\tau} + \alpha^2L}{\mu} \right)$ in contrast to $\cO\left(\frac{\beta + \kappa\alpha^2L}{\mu} \right)$. If the first term is dominating, then it implies an improvement proportional to the batch size $\tau$. If, in contrast, the second term is larger, then the improvement is even more significant and the guarantee is $\cO(\kappa)$ times better, which is often a very large constant.

The proof technique for this theorem also uses recent advances on the analysis of biased SGD methods by \citet{mishchenko2020random}. In particular, we show that the three-point identity (provided in the Appendix) is useful for getting a tighter recursion.

Next, we extend this result to the nonconvex convergence as given under the following assumption on bounded variance.
\begin{assumption}\label{as:bounded_var}
	We assume that the variance of meta-loss gradients is uniformly bounded by some $\sigma^2$, i.e.,
	\begin{align}
		\E{\|\nabla F_i(x) - \nabla F(x)\|^2}
		\le \sigma^2. \label{eq:bounded_var}
	\end{align}
\end{assumption}
The new assumption on bounded variance is different from the one we used previously of variance being finite at the optimum, which was given in equation~\eqref{eq:def_sigma}. At the same time, it is very common in literature on stochastic optimization when studying convergence on nonconvex functions.
\begin{theorem}\label{th:nonconvex_fo_maml}
	Let \Cref{as:bounded_var} hold, functions $f_1,\dotsc, f_n$ be $L$--smooth and $F$ be lower bounded by $F^*>-\infty$. Assume $\alpha\le \frac{1}{4L}, \beta\le \frac{1}{16L}$. If we consider the iterates of \Cref{alg:fo_maml} (with $\delta=\alpha L$) or \Cref{alg:mamlP} (with general $\delta$), then
	\begin{align*}
		\min_{t\le k}\E{\|\nabla F(x^t)\|^2}
		&\le \frac{4}{\beta k}\E{F(x^0)-F^*} + 4(\alpha L)^2\delta^2 \sigma^2 
		+ 32 \beta(\alpha L)^2 \left(\frac{1}{|T_k|} + (\alpha L)^2\delta^2\right) \sigma^2.
	\end{align*}
\end{theorem}
Notice that this convergence is also only until some neighborhood of first-order stationarity, since the second term does not decrease with $k$. This size of the upper bound depends on the product $\cO((\alpha L)^2 \delta^2)$, so to obtain better convergence one can simply increase approximation accuracy to make $\delta$ smaller. However, the standard FO-MAML corresponds to $\delta=\alpha L$, so its convergence guarantees directly depend on the problem parameter $\alpha$.

For Algorithm~\ref{alg:maml2}, we have $\delta=\cO((\alpha L)^s)$ as per Lemma~\ref{lem:approx_implicit}, and we recover convergence guarantee up to a neighborhood of size $\cO((\alpha L)^2\delta^2)=\cO((\alpha L)^{2s+2})$. Therefore, to make this smaller than some given target accuracy $\varepsilon>0$, we need at most $s=\cO(\log\frac{1}{\varepsilon})$ inner-loop iterations. If we can plug-in $s=1$, we also get that FO-MAML converges to a neighborhood of size $\cO((\alpha L)^4)$.

Our \Cref{th:nonconvex_fo_maml} is very similar to the one obtained by \citet{fallah2020convergence}, except their convergence neighborhood depends on $\alpha$ as $\cO(\alpha^2)$, whereas ours is of size $\cO(\alpha^4)$, which goes to 0 much faster when $\alpha\to 0$. Moreover, in contrast to their theory, ours does not require any assumptions on the Hessian smoothness. Note, in addition, that the main difference comes from the kind of objectives that we study, as \citet{fallah2020convergence} considered minimization of problems not involving Moreau envelopes.


\section{Conclusion}
In this paper, we presented a new analysis of first-order meta-learning algorithms for minimization of Moreau envelopes. Our theory covers both nonconvex and strongly convex smooth losses and guarantees convergence of the family of methods covered by Algorithm~\ref{alg:mamlP}. As a special case, all convergence bounds apply to Algorithm~\ref{alg:maml2} with an arbitrary number of inner-loop steps. Compared to other results available in the literature, ours are more general as they hold with an arbitrary number of inner steps and do not require Hessian smoothness. The main theoretical difficulty we faced was the limitation of the inexact SGD framework, which we overcame by presenting a refined analysis using virtual iterates. As a minor contribution, we also pointed out that standard algorithms, such as SGD, are not immediately guaranteed to work on the iMAML objective, which might be nonconvex and nonsmooth even for convex and smooth losses. To show this, we presented examples of losses whose convexity and smoothness cease when the iMAML objective is constructed.

\bibliographystyle{apalike}
\bibliography{refs}
\appendix
\onecolumn

\section{Content left out}
\subsection*{Table of frequently used notation}
For clarity, we provide a table of frequently used notation.
\begin{table}[h]
	\centering
	\begin{tabular}{cc}
		\toprule
		\textbf{Notation} & \textbf{Meaning}\\
		\midrule
		$f_i$ & The loss of task $i$ \\
		$F_i(x) = \min_z \{f_i(z) + \frac{1}{2\alpha}\|z-x\|^2\}$ & Meta-loss\\
		$F(x)=\frac{1}{n}\sum_{i=1}^n F_i(x)$ & Full meta loss\\
		$z_i(x)=\argmin_z \{f_i(z) + \frac{1}{2\alpha}\|z-x\|^2\}$  & The minimizer of regularized loss\\
		\hline
		$L, \mu$ & Smoothness and strong convexity constants of $f_i$\\
		$\Lmeta$ & Smoothness constant of $\meta$\\
		$\alpha$ & Objective parameter\\
		$\outers$ & Stepsize of the outer loop\\
		$\locstep, s$ & Stepsize and number of steps in the inner loop\\
		$\delta$ & Precision of the proximal oracle\\
		\bottomrule
	\end{tabular}
\end{table}
\vspace{-1pt}

\subsection{Parametrization of the inner loop of \Cref{alg:maml2}}

Note that \Cref{alg:maml2} depends on only one parameter -- $\outers$. We need to keep in mind that parameter $\inners$ is fixed by the objective \eqref{eq:new_pb} and changing $\inners$ shifts convergence neighborhood. Nevertheless, we can still investigate the case wehn $\inners$ from \eqref{eq:new_pb} and $\inners$ from Line $6$ of \Cref{alg:maml2} are different, as we can see in the following remark. 
\begin{remark} \label{lem:locstep_different}
	If we replace line $6$ of \Cref{alg:maml2} by $z_{l+1}^k = x^k - \locstep \nabla f_i(z_{i,l}^k)$, we will have freedom to choose $\locstep$. However, if we choose stepsize $\locstep \neq \inners$, then similar analysis to the proof of \Cref{lem:approx_implicit} yields 
	\begin{align} \label{eq:locstep_inexact}
	\frac 1 \locstep \|z_{i,s}^k - (x^k - \locstep\nabla \meta_i(x^k))\|
	&\leq  \left( (\locstep L)^s + |\alpha - \locstep| L \right)  \| \nabla \meta_i(x^k)\|.
	\end{align}
\end{remark}
	Note that in case $\locstep \neq \inners$, we cannot set number of steps $s$ to make the right-hand side of~\eqref{eq:locstep_inexact} smaller than $\delta \| \nabla \meta_i(x^k)\|$ when $\delta$ is small. In particular, increasing the number of local steps $s$ will help only as long as $\delta > |\inners - \locstep|L$.
	
	This is no surprise, for the modified algorithm (using inner loop stepsize $\locstep$) will no longer be approximating $\nabla \meta_i(x^k)$. It will be exactly approximating $\nabla \tilde \meta_i(x^k)$, where  $\tilde \meta_i(x)\eqdef \min_{z\in\R^d} \left\{f_i(z) + \frac{1}{2\locstep}\|z - x\|^2\right\}$ (see \Cref{lem:approx_implicit}).
	Thus, choice of stepsize in the inner loop affects what implicit gradients do we approximate and also what objective we are minimizing.s

\section{Proofs}
\subsection{Basic facts}
For any vectors $a, b\in\R^d$ and scalar $\nu>0$, Young's inequality states that
\begin{align}
    2\<a, b>\le \nu\|a\|^2+ \frac{1}{\nu}\|b\|^2. \label{eq:young}
\end{align}
Moreover, we have
\begin{align}
    \|a+b\|^2
    \le 2\|a\|^2 + 2\|b\|^2.\label{eq:a_plus_b}
\end{align}
More generally, for a set of $m$ vectors $a_1,\dotsc, a_m$ with arbitrary $m$, it holds
\begin{align}
    \Bigl\|\frac{1}{m}\sum_{i=1}^m a_i\Bigr\|^2
    \le \frac{1}{m}\sum_{i=1}^m\|a_i\|^2. \label{eq:jensen_for_sq_norms}
\end{align}
For any random vector $X$ we have
\begin{align}
    \ec{\|X\|^2}=\|\ec{X}\|^2 + \ec{\|X-\ec{X}\|^2}. \label{eq:rand_vec_sq_norm}
\end{align}
If $f$ is $L_f$-smooth, then for any $x,y\in\R^d$, it is satisfied
\begin{equation}
	f(y)
	\le f(x) + \<\nabla f(x), y-x> + \frac{L_f}{2}\|y-x\|^2. \label{eq:smooth_func_vals}
\end{equation}

Finally, for $L_f$-smooth and convex function $f$, it holds
\begin{equation}
f(x) \leq f(y)+\langle\nabla f(x), x-y\rangle-\frac{1}{2 L_f}\|\nabla f(x)-\nabla f(y)\|^{2} \label{eq:smooth_conv}.
\end{equation}

\begin{proposition} \label{pr:three_point} [Three-point identity]
	For any $u,v,w \in \R^d$, any $f$ with its Bregman divergence $D_f(x,y) = f(x)-f(y) - \langle \nabla f(y), x-y\rangle $, it holds
	\begin{align*}
	\langle \nabla f(u) - \nabla f(v), w-v \rangle
	&= D_f (v,u) + D_f(w,v) - D_f(w,u).
	\end{align*}
\end{proposition}

\subsection{Proof of \Cref{th:imaml_nonconvex}}
\begin{proof}
	The counterexample that we are going to use is given below:
	\begin{align*}
		f(x) &= \min\left\{\frac{1}{4}x^4 - \frac{1}{3}|x|^3 + \frac{1}{6}x^2, \frac{2}{3}x^2 - |x| + \frac{5}{12} \right\} \\
		&=
		\begin{cases}
			\frac{1}{4}x^4 - \frac{1}{3}|x|^3 + \frac{1}{6}x^2, &\textrm{if } |x|\le 1,\\
			\frac{2}{3}x^2 - |x| + \frac{5}{12}, & \textrm{otherwise}.
		\end{cases}
	\end{align*}
	See also Figure~\ref{fig:example} for its numerical visualization.
	\begin{figure}[h!]
	\minipage[t]{0.48\textwidth}
	\centering
	\includegraphics[width=0.95\linewidth]{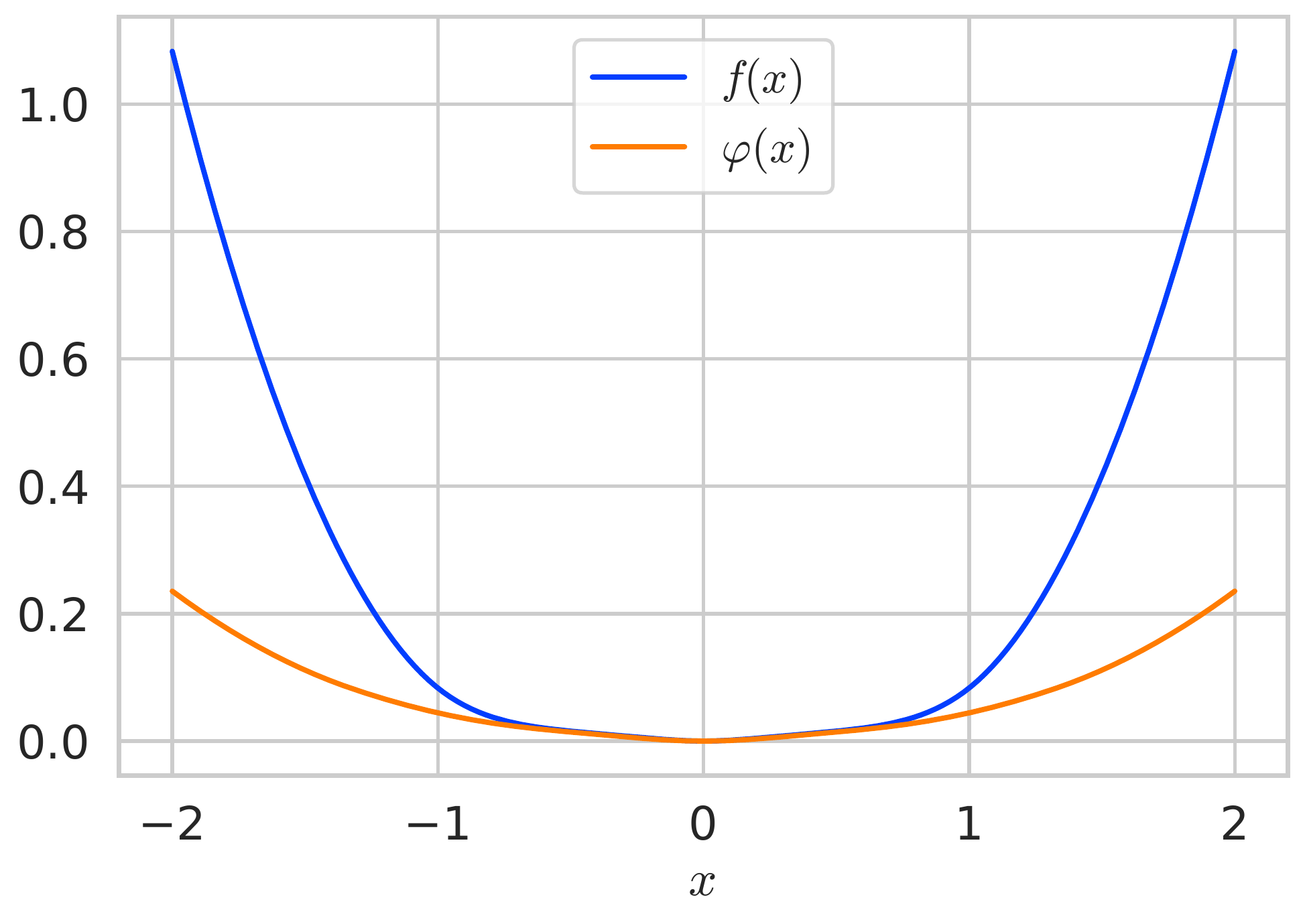}
	\caption{Values of functions $f$ and $\varphi$.}
	\label{fig:example}
	\endminipage\hspace{0.2cm}
	\minipage[t]{0.48\textwidth}
	\centering
	\includegraphics[width=0.95\linewidth]{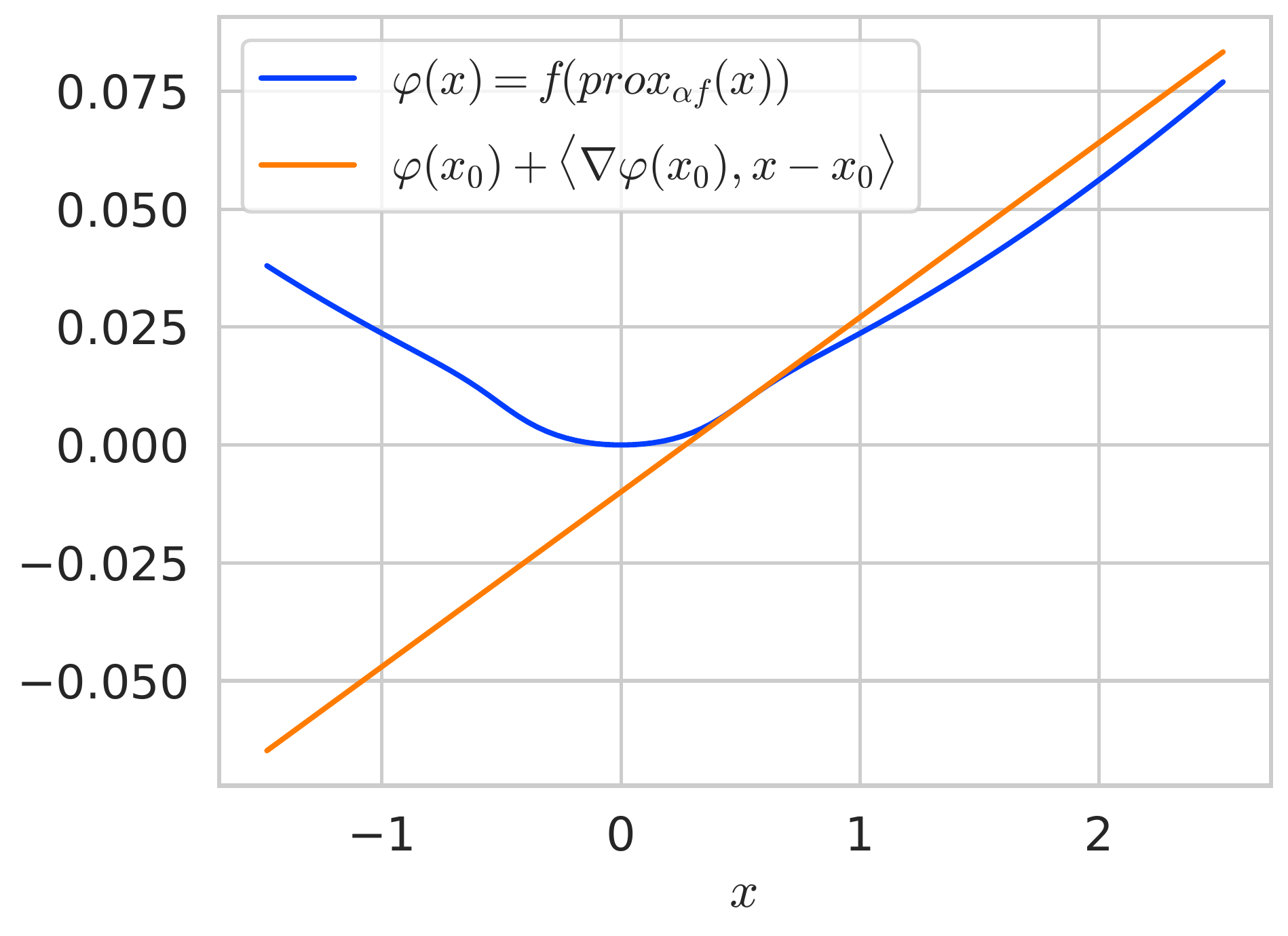}
	\caption{Illustration of nonconvexity: the value of $\varphi$ goes below its tangent line from $x_0$, which means that $\varphi$ is nonconvex at $x_0$.}
	\endminipage
	\end{figure}
	
	It is straightforward to observe that this function is smooth and convex because its Hessian is
	\[
		f''(x)
		= \begin{cases}
			3x^2 - 2|x| + \frac{1}{3}, &\textrm{if } |x|\le 1,\\
			\frac{4}{3}, & \textrm{otherwise}.
		\end{cases},
	\]
	which is always nonnegative and bounded. However, the function $\varphi(x) = f(z(x))$ is not convex at point $x_0=0.4 + \alpha\nabla f(0.4)$, because its Hessian is negative, i.e., $\varphi''(x_0)<0$, which we shall prove below. First of all, by definition of $x_0$, it holds that $0.4=x_0 - \alpha\nabla f(0.4)$, which is equivalent to the definition of $z(x)$, implying $z(x_0)=0.4$. Next, let us obtain the expression for the Hessian of $\varphi$. As shown in \cite{rajeswaran2019meta}, it holds in general that
	\[
		\nabla \varphi(x)
		= \frac{dz(x)}{dx}\nabla f(z(x)),
	\]
	where $\frac{dz(x)}{dx}$ is the Jacobian matrix of the mapping $z(x)$. Differentiating this equation again, we obtain
	\[
		\nabla^2 \varphi(x)
		= \frac{d^2z(x)}{dx^2}\nabla f(z(x)) + \nabla^2 f(z(x))\frac{dz(x)}{dx}\Bigl(\frac{dz(x)}{dx}\Bigr)^\top.
	\]
	Moreover, we can compute $\frac{d^2z(x)}{dx^2}$ by differentiating two times the equation $z(x) = x - \alpha \nabla f(z(x))$, which gives
	\[
		\frac{dz(x)}{dx} = \mathbf{I} - \alpha \nabla^2 f(z(x))\frac{dz(x)}{dx},
	\]
	where $\mathbf{I}$ is the identity matrix. Rearranging the terms in this equation yields 
	\[
		\frac{dz(x)}{dx} = (\mathbf{I} + \alpha \nabla^2 f(z(x)))^{-1}.
	\]
	At the same time, if we do not rearrange and instead differentiate the equation again, we get
	\[
		\frac{d^2z(x)}{dx^2}
		= -\alpha \nabla^2 f(z(x))\frac{d^2z(x)}{dx^2} - \alpha \nabla^3 f(z(x)) \left[\frac{dz(x)}{dx}, \frac{dz(x)}{dx}\right],
	\]
	where $\nabla^3 f(z(x))[\frac{dz(x)}{dx},\frac{dz(x)}{dx}]$ denotes tensor-matrix-matrix product, whose result is a tensor too. Thus,
	\[
		\frac{d^2z(x)}{dx^2}
		= -\alpha (\mathbf{I} + \alpha\nabla^2 f(z(x)))^{-1}\nabla^3 f(z(x))\left[\frac{dz(x)}{dx}, \frac{dz(x)}{dx}\right],
	\]
	and, moreover,
	\[
		\nabla^2 \varphi(x)
		= -\alpha(\mathbf{I} + \alpha\nabla^2 f(z(x)))^{-1}\nabla^3 f(z(x))\left[\frac{dz(x)}{dx}, \frac{dz(x)}{dx}\right] + \nabla^2 f(z(x))\frac{dz(x)}{dx}\Bigl(\frac{dz(x)}{dx}\Bigr)^\top.
	\]	
	For any $x\in(0, 1]$, our counterexample function satisfies $f''(x)=3x^2 - 2x + \frac{1}{3}$ and $f'''(x)=6x-2$. Moreover, since $z(x_0)=0.4$, we have $f''(z(x_0))=\frac{1}{75}$,  $f'''(z(x_0))=\frac{2}{5}$, $\frac{dz(x)}{dx}= \frac{1}{1+ \alpha/75}$, and 
	\[
		\varphi''(x)=- \frac{2\alpha}{5(1+ \alpha/75)^3} + \frac{1}{75(1+ \alpha/75)^2}.
	\]
	It can be verified numerically that $\varphi''(x)$ is negative at $x_0$ for any $\alpha>\frac{75}{2249}$. Notice that this value of $\alpha$ is much smaller than the value of $\frac{1}{L}=\frac{3}{4}$, which can be obtained by observing that our counterexample satisfies $f''(x)\le \frac{4}{3}$.
\end{proof}
Let us also note that obtaining nonconvexity of this objective for a fixed function and arbitrary $\alpha$ is somewhat challenging. Indeed, in the limit case $\alpha\to 0$, it holds that $\varphi(x)''\to f''(x)$ for any $x$. If $f''(x)>0$ then for a sufficiently small $\alpha$ it would also hold $\varphi''(x)>0$. Finding an example that works for any $\alpha$, thus, would require $f''(x_0)=0$.

\subsection{Proof of \Cref{th:imaml_nonsmooth}}
\begin{proof}
	Consider the following simple function
	\[
		f(x) = \frac{1}{2}x^2 + \cos(x).
	\]
	The Hessian of $f$ is $f''(x)=1 - \cos(x)\ge 0$, so it is convex. Moreover, it is apparent that the gradient and the Hessian of $f$ are Lipschitz. However, we will show that the Hessian of $\varphi$ is unbounded for any fixed $\alpha>0$. To establish this, let us first derive some properties of $z(x)$. First of all, by definition $z(x) $ is the solution of $\alpha f'(z(x))+(z(x)-x)=0$, where by definition of $f$, it holds $f'(z(x))=z(x) - \sin(z(x))$. Plugging it back, we get
	\[
		(\alpha + 1) z(x) - \alpha \sin (z(x)) = x.
	\]
	Differentiating both sides with respect to $x$, we get $(\alpha+1)\frac{d z(x)}{dx} - \alpha \cos(z(x))\frac{d z(x)}{dx} = 1$ and 
	\[
		\frac{dz(x)}{dx} = \frac{1}{1 + \alpha - \alpha \cos(z(x))}.
	\]	
	Thus, using the fact that $\varphi(x)=\varphi(z(x))$, we get
	\[
		\varphi'(x)
		= \frac{d\varphi(x)}{dx} 
		= \frac{df(z)}{dz}\frac{d z(x)}{dx}
		= \frac{z(x) - \sin(z(x)) }{1 + \alpha - \alpha \cos(z(x))}.
	\]
	Denoting, for brevity, $z(x)$ as $z$, we differentiate this identity with respect to $z$ and derive $\frac{d\varphi'(x)}{dz} = \frac{1 + 2\alpha - \alpha z \sin(z) - (1+2\alpha)\cos(z)}{(1+\alpha - \alpha\cos(z))^2}$.
	Therefore, for the Hessian of $\varphi$, we can produce an implicit identity,
	\[
		\varphi''(x)
		= \frac{d^2\varphi(x)}{dx^2} 
		= \frac{d\varphi'(x)}{dz}\frac{dz(x)}{dx}
		= \frac{1 + 2\alpha - \alpha z \sin(z) - (1+2\alpha)\cos(z)}{(1+\alpha - \alpha\cos(z))^3}.
	\]	
	The denominator of $\varphi''(x)$ satisfies $|1+\alpha - \alpha\cos(z)|^3\le (1 + 2\alpha)^3$, so it is bounded for any $x$. The numerator, on the other hand, is unbounded in terms of $z(x)$ since $|1 + 2\alpha - \alpha z \sin(z) - (1+2\alpha)\cos(z)|\ge \alpha |z\sin(z)| - 2(1+2\alpha)$. Therefore, $|\varphi''(x)|$ is unbounded. Moreover, $z(x)$ is itself unbounded, since the previously established identity for $z(x)$ can be rewritten as $|z(x)|=\left|\frac{1}{1+\alpha}x - \frac{\alpha}{1+\alpha}\sin(z(x))\right|\ge \frac{1}{1+\alpha}|x| - 1$. Therefore, $z(x)$ is unbounded, and since $\varphi''(x)$ grows with $z$, it is unbounded too. The unboundedness of $\varphi''(x)$ implies that $\varphi$ is not $L$-smooth for any finite $L$.
\end{proof}

\subsection{Proof of \Cref{lem:moreau_is_str_cvx_and_smooth}}
\begin{proof}
	The statement that $\meta_i$ is $\frac{\mu}{1+\alpha\mu}$-strongly convex is proven as Lemma 2.19 in \cite{planiden2016strongly}, so we skip this part.
	
	For nonconvex $\meta_i$ and any $x\in\R^d$, we have by first-order stationarity of the inner problem that $\nabla \meta_i(x) = \nabla f_i(z_i(x))$, where $z_i(x)=\arg\min_z \{f_i(z) + \frac{1}{2\alpha}\|z-x\|^2\} = x - \alpha \nabla \meta_i(x)$. Therefore,
	\begin{align*}
		\|\nabla \meta_i(x) - \nabla \meta_i(y)\|
		&= \|\nabla f_i(z_i(x)) - \nabla f_i(z_i(y))\|
		\le L \|z_i(x) - z_i(y)\| \\
		&= L\|x-y-\alpha(\nabla \meta_i(x) - \nabla \meta_i(y))\| \\
		&\le L\|x-y\| + \alpha L \|\nabla \meta_i(x) - \nabla \meta_i(y)\|.
	\end{align*}
	Rearranging the terms, we get the desired bound:
	\[
		\|\nabla \meta_i(x) - \nabla \meta_i(y)\|
		\le \frac{L}{1-\alpha L}\|x-y\|.
	\]
	
	For convex functions, our proof of smoothness of $\meta_i$ follows the exact same steps as the proof of Lemma 2.19 in \cite{planiden2016strongly}. Let $f_i^*$ be the convex-conjugate of $f_i$. Then, it holds that $\meta_i = (f_i^* + \frac{\alpha}{2}\|\cdot\|^2)^*$, see Theorem 6.60 in~\cite{beck-book-first-order}. Therefore, $\meta_i^* = f_i^* + \frac{\alpha}{2}\|\cdot\|^2$. Since $f_i$ is $L$-smooth, $f_i^*$ is $\frac{1}{L}$-strongly convex. Therefore, $\meta_i^*$ is $(\frac{1}{L} + \alpha)$-strongly convex, which, finally, implies that $\meta_i$ is $\frac{1}{\frac{1}{L}+\alpha}$-smooth.
	
	The statement $\frac{L}{1+\alpha L}\le L$ holds trivially since $\alpha>0$. In case $\alpha\le \frac{1}{\mu}$, we get the constants from the other statements by mentioning that $\frac{\mu}{1+\alpha\mu}\ge \frac{\mu}{2}$.
	
	The differentiability of $\meta_i$ follows from Theorem 4.4 of \citet{poliquin1996prox}, who show differentiability assuming $f_i$ is \emph{prox-regular}, which is a strictly weaker property than $L$-smoothness, so it automatically holds under the assumptions of \Cref{lem:moreau_is_str_cvx_and_smooth}.
\end{proof}

\subsection{Proof of \Cref{lem:approx_implicit}}
\textbf{\Cref{lem:approx_implicit}}.
	Let task losses $f_i$ be $L$--smooth and $\alpha>0$. Given $i$ and $x\in\R^d$, we define recursively $z_{i,0}=x$ and $z_{i,j+1} = {\color{blue}x} - \alpha \nabla f_i({\color{mydarkred}z_{i,j}})$. Then, it holds for any $s\ge 0$
	\begin{align*}
		\left\| \nabla f_i(z_{i,s}) - \nabla \meta_i(x) \right\|
		\le (\alpha L)^{s+1} \|\nabla \meta_i(x)\|.
	\end{align*}
	In particular, the iterates of FO-MAML (\Cref{alg:fo_maml}) satisfy for any $k$
	\begin{align*}
		\left\| \nabla f_i(z_i^k) - \nabla \meta_i(x^k) \right\|
		\le (\alpha L)^2 \|\nabla \meta_i(x^k)\|.
	\end{align*}

\begin{proof}
	First, observe that by \cref{eq:implicit} it holds 
	\begin{align*}
	z_i(x)=x - \alpha \nabla\meta_i(x)=x-\alpha \nabla f_i(z_i(x)).
	\end{align*}
	For $s=0$, the lemma's claim then follows from initialization, $z_{i, 0}=x$, since 
	\[
		\|\nabla f_i(z_{i,s}) - \nabla \meta_i(x)\| = \|\nabla f_i(x) - \nabla f_i(z_i(x))\|
		\le L \|x - z_i(x)\|
		= \alpha L \|\nabla \meta_i(x)\|.
	\]
	For $s>0$, we shall prove the bound by induction.
	We have for any $l\ge 0$
	\begin{align*}
	\|z_{i,l+1} - (x - \alpha\nabla \meta_i(x))\|
	&= \alpha \|\nabla f_i (z_{i,l}) - \nabla \meta_i(x)\|
	= \alpha \|\nabla f_i (z_{i,l}) - \nabla f_i(z_i(x))\|
	\le \alpha L\|z_{i, l} - z_i(x)\| \\
	&= \alpha L\|z_{i, l} - (x - \alpha\nabla \meta_i(x))\|.
	\end{align*}
	This proves the induction step as well as the lemma itself.
\end{proof}

\begin{lemma}\label{lem:maml_approx_grad}
	If task losses $f_1,\dotsc, f_n$ are $L$-smooth and $\outers \le \frac{1}{L}$, then it holds
	\begin{align}
	\Bigl\|\frac{1}{|T_k|}\sum_{i\in T_k} g_i^k\Bigr\|^2
	&\le \left(1 + 2 (\inners L)^{2s} + \frac 2 {|T|} \right)4L(\meta(x^k) - \meta(\opt)) + 4\left(\frac 1 {|T_k|} + (\inners L)^{2s}\right)\sigma_*^2\\
	&\le 20L(\meta(x^k) - \meta(\opt)) + 4\left(\frac 1 {|T_k|} + \delta^2\right)\sigma_*^2. \label{eq:ave_grad_norm}
	\end{align}
\end{lemma}

\begin{proof}
	First, let us replace $g_i^k$ with $\nabla \meta_i(x^k)$, which $g_i^k$ approximates:
	\begin{align*}
	\Bigl\|\frac{1}{|T_k|}\sum_{i\in T_k} g_i^k\Bigr\|^2
	&= \Bigl\|\frac{1}{|T_k|}\sum_{i\in T_k} \nabla\meta_i(x^k) + \frac{1}{|T_k|}\sum_{i\in T_k} (g_i^k - \nabla\meta_i(x^k))\Bigr\|^2 \\
	&\overset{\eqref{eq:a_plus_b}}{\le} 2\Bigl\|\frac{1}{|T_k|}\sum_{i\in T_k} \nabla\meta_i(x^k)\Bigr\|^2 + 2\Bigl\| \frac{1}{|T_k|}\sum_{i\in T_k} (g_i^k - \nabla\meta_i(x^k))\Bigr\|^2 \\
	&\overset{\eqref{eq:jensen_for_sq_norms}}{\le} 2\Bigl\|\frac{1}{|T_k|}\sum_{i\in T_k} \nabla\meta_i(x^k)\Bigr\|^2 + \frac{2}{|T_k|}\sum_{i\in T_k}\| g_i^k - \nabla\meta_i(x^k)\|^2 \\
	&\le 2\Bigl\|\frac{1}{|T_k|}\sum_{i\in T_k} \nabla\meta_i(x^k)\Bigr\|^2 + \frac{2}{|T_k|}\sum_{i\in T_k} \delta^2 \| \nabla\meta_i(x^k)\|^2.
	\end{align*}
	Taking the expectation on both sides, we get
	\begin{align*}
	\ec{\Bigl\|\frac{1}{|T_k|}\sum_{i\in T_k} g_i^k\Bigr\|^2}
	\overset{\eqref{eq:rand_vec_sq_norm}}{\le} 2\|\nabla\meta(x^k)\|^2 + 2\ec{\Bigl\| \frac{1}{|T_k|}\sum_{i\in T_k}\nabla\meta_i(x^k) - \nabla\meta(x^k)\Bigr\|^2} + \frac{2}{n}\sum_{i=1}^n \delta^2\| \nabla\meta_i(x^k)\|^2.
	\end{align*}
	Moreover, each summand in the last term can be decomposed as
	\begin{align*}
	\| \nabla\meta_i(x^k)\|^2
	\overset{\eqref{eq:a_plus_b}}{\le} 2\| \nabla\meta_i(\opt)\|^2 + 2\|\nabla \meta_i(x^k) - \nabla \meta_i(\opt)\|^2
	\overset{\eqref{eq:def_sigma}}{=} 2\sigma_*^2 + 2\|\nabla \meta_i(x^k) - \nabla \meta_i(\opt)\|^2.
	\end{align*}
	Since $\meta_i$ is convex and $L$-smooth, we have for any $i$
	\begin{align*}
	\|\nabla \meta_i(x^k) - \nabla \meta_i(\opt)\|^2 \le 2L(\meta_i(x^k) - \meta_i(\opt) - \<\nabla \meta_i(\opt), x^k-\opt>).
	\end{align*}
	Averaging and using $\frac{1}{n}\sumin \nabla \meta_i(\opt)=0$, we obtain
	\begin{align*}
	\frac{1}{n}\sum_{i=1}^n\| \nabla\meta_i(x^k)-\nabla\meta_i(\opt)\|^2
	\le 2L(\meta(x^k) - \meta(\opt)).
	\end{align*}
	Thus,
	\begin{align}
	\frac{2}{n}\sum_{i=1}^n \delta^2 \| \nabla\meta_i(x^k)\|^2
	&\le 4 \delta^2 \sigma_*^2 + 8L \delta^2 (\meta(x^k)-\meta(\opt)) \label{eq:meta_grad_norm} \\
	&\le 4 \delta^2 \sigma_*^2 + 8L(\meta(x^k)-\meta(\opt)). \notag
	\end{align}
	Proceeding to another term in our initial bound, by independence of sampling $i\in T_k$ we have
	\begin{align*}
	\ec{\Bigl\| \frac{1}{|T_k|}\sum_{i\in T_k}\nabla\meta_i(x^k) - \nabla\meta(x^k)\Bigr\|^2}
	&= \frac{1}{|T_k|} \avein \ec{\|\nabla\meta_i(x^k)\|^2} \\
	&\overset{\eqref{eq:a_plus_b}}{\le} \frac{2}{|T_k|} \avein \left(\ec{\|\nabla\meta_i(x^k) - \nabla\meta_i(\opt)\|^2} + \ec{\|\nabla \meta_i(\opt)\|^2} \right) \\
	&\overset{\eqref{eq:smooth_conv}}{\leq} \frac 2 {|T_k|} \left(2L (\meta(x^k) - \meta(\opt)) + \sigma_*^2 \right)\\
	&\le \frac {4L} {|T_k|} (\meta(x^k) - \meta(\opt)) + \frac 2 {|T_k|} \sigma_*^2.
	\end{align*}
	Finally, we also have $\|\nabla \meta(x^k)\|^2\le 2L(\meta(x^k) - \meta(\opt))$. Combining all produced bounds, we get the claim
	\begin{align}
	\Bigl\|\frac{1}{|T_k|}\sum_{i\in T_k} g_i^k\Bigr\|^2
	\le \left(1 + 2  \delta^2 + \frac 2 {|T|} \right)4L(\meta(x^k) - \meta(\opt)) + 4\left(\frac 1 {|T_k|} + \delta^2 \right)\sigma_*^2.
	\end{align}
\end{proof}

\subsection{Proof of \Cref{th:convergence_of_mamlP}}
\textbf{\Cref{th:convergence_of_mamlP}}.
	Let task losses $f_1,\dotsc, f_n$ be $L$-smooth and $\mu$-strongly convex. If $|T_k|=\tau$ for all $k$, $\alpha\le\frac{1}{L}, \outers \leq \frac 1 {20L}$ and $\delta \leq \frac 1 {4 \sqrt \kappa}$, where $\kappa\eqdef \frac{L}{\mu}$, then the iterates of \Cref{alg:mamlP} satisfy
\begin{align*}
	\ec{\|x^k-\opt\|^2}
	\le \left(1 - \frac{\outers\mu}{4}\right)^k\|x^0-\opt\|^2 + \frac{16}{\mu} 	\left( \frac {2\delta^2} \mu + \frac \outers {\tau} + \outers  \delta^2  \right) \sigma_*^2.
\end{align*}

\begin{proof}
	For the iterates of \Cref{alg:mamlP}, we can write
	\begin{align*}
	x^{k+1} 
	= x^k - \frac{\outers}{\tau}\sum_{i\in T_k} g_i^k.
	\end{align*}
	We also have by \Cref{lem:approx_implicit} that 
	\[
		\|g_i^k - \nabla \meta_i(x^k)\|^2 \le (\alpha L)^2\delta^2 \|\nabla \meta_i (x^k)\|^2 \le \delta^2 \|\nabla \meta_i (x^k)\|^2,
	\]
	so let us decompose $g_i^k$ into $\nabla \meta_i(x^k)$ and the approximation error:
	\begin{align*}
	\|x^{k+1} - \opt\|^2
	&= \|x^k - \opt\|^2 - \frac{2\outers}{\tau}\sum_{i\in T_k} \<g_i^k, x^k - \opt> + \outers^2\Bigl\|\frac{1}{\tau}\sum_{i\in T_k} g_i^k\Bigr\|^2 \\
	&= \|x^k - \opt\|^2 - \frac{2\outers}{\tau}\sum_{i\in T_k} \<\nabla\meta_i(x^k), x^k - \opt> + \frac{2\outers}{\tau}\sum_{i\in T_k} \<\nabla\meta_i(x^k) - g_i^k, x^k - \opt> + \outers^2\Bigl\|\frac{1}{\tau}\sum_{i\in T_k} g_i^k\Bigr\|^2.
	\end{align*}
	First two terms can be upperbounded using strong convexity (recall that by \Cref{lem:moreau_is_str_cvx_and_smooth}, $\meta_i$ is $\frac{\mu}{2}$-strongly convex):
	\begin{align*}
	\|x^k - \opt\|^2 - \frac{2\outers}{\tau}\sum_{i\in T_k} \<\nabla\meta_i(x^k), x^k - \opt>
	&\le \left(1 - \frac{\outers\mu}{2}\right)\|x^k - \opt\|^2 - \frac{2\outers}{\tau}\sum_{i\in T_k}(\meta_i(x^k) - \meta_i(\opt)).
	\end{align*}
	For the third term, we will need Young's inequality:	
	\begin{align*}
	2\<\nabla\meta_i(x^k) - g_i^k, x^k - \opt>
	&\overset{\eqref{eq:young}}{\le}\frac{4}{\mu}\|\nabla\meta_i(x^k) - g_i^k \|^2 + \frac{\mu}{4}\|x^k - \opt\|^2 \le\frac{4}{\mu} \delta^2 \|\nabla\meta_i(x^k)\|^2 + \frac{\mu}{4}\|x^k - \opt\|^2,
	\end{align*}
	which we can scale by $\beta$ and average over $i\in T_k$ to obtain
	\[
	\frac{2\outers}{\tau}\sum_{i\in T_k} \<\nabla\meta_i(x^k) - g_i^k, x^k - \opt>
	 \leq  \frac{4 \outers \delta^2}{\mu} \frac 1 \tau \sum_{i \in T_k} \|\nabla \meta_i(x^k)\|^2 + \frac{\beta \mu}{4}\|x^k - \opt\|^2.
	\]

	Plugging in upper bounds and taking expectation yields
	\begin{align*}
	\ec{\|x^{k+1}-\opt\|^2} 
	&\le \left(1 - \frac{\outers\mu}{4}\right)\|x^k-\opt\|^2 - 2\outers (\meta(x^k) - \meta(\opt)) + \frac{4}{\mu}\outers \delta^2 \avein \|\nabla \meta_i(x^k)\|^2 + \outers^2\Bigl\| \frac{1}{\tau}\sum_{i\in T_k} g_i^k\Bigr\|^2  \\
	&\overset{\eqref{eq:ave_grad_norm}}{\le} \left(1 - \frac{\outers\mu}{4}\right)\|x^k-\opt\|^2 - 2\outers(1 - 10\beta L)  (\meta(x^k) - \meta(\opt)) + \frac{4}{\mu}\outers  \delta^2 \avein\|\nabla \meta_i(x^k)\|^2 \\
	&\quad + 4\outers^2 \left(\frac 1 {\tau} + \delta^2 \right)\sigma_*^2 \\
	&\overset{\eqref{eq:meta_grad_norm}}{\le} \left(1 - \frac{\outers\mu}{4}\right)\|x^k-\opt\|^2 - 2\outers(1 - 10\outers L)  (\meta(x^k) - \meta(\opt))\\
	&\quad + \frac{8}{\mu}\outers \delta^2\left(\sigma_*^2 + 2L(\meta(x^k) - \meta(\opt))\right) + 4\outers^2 \left(\frac 1 {\tau} + \delta^2\right)\sigma_*^2 \\
	&= \left(1 - \frac{\outers\mu}{4}\right)\|x^k-\opt\|^2 - 2\outers \left(1 - 10\outers L - \frac {8L} \mu  \delta^2 \right)  (\meta(x^k) - \meta(\opt)) + \frac{8}{\mu}\outers \delta^2 \sigma_*^2 + 4\outers^2 \left(\frac 1 {\tau} + \delta^2 \right)\sigma_*^2.
	\end{align*}
	By assumption $\outers \leq \frac 1 {20L}, \delta \leq \frac 1 {4\sqrt{ \kappa}}$, we have $10\outers L \leq \frac 1 2$ and $8 \frac{L}{\mu} \delta^2\leq  \frac 1 2$, so $1 - 10\outers L - \frac {8L} \mu  \delta^2  \geq 0$, hence
	\begin{align*}
	\ec{\|x^{k+1}-\opt\|^2} 
	&\le \left(1 - \frac{\outers\mu}{4}\right)\|x^k-\opt\|^2 + \frac{8}{\mu}\outers \delta^2 \sigma_*^2 + 4\outers^2 \left(\frac 1 {\tau} +  \delta^2 \right)\sigma_*^2.
	\end{align*}
	Recurring this bound, which is a standard argument, we obtain the theorem's claim.
	\begin{align*}
	\ec{\|x^{k}-\opt\|^2} 
	&\le \left(1 - \frac{\outers\mu}{4}\right)^k\|x^0-\opt\|^2 + \left( \frac{8}{\mu}\outers \delta^2 \sigma_*^2 + 4\outers^2 \left(\frac 1 {\tau} + \delta^2\right)\sigma_*^2 \right) \frac {1-\left( 1- \frac {\outers \mu} 4 \right)^k}{\frac {\outers \mu} 4}\\
	&\le \left(1 - \frac{\outers\mu}{4}\right)^k\|x^0-\opt\|^2 + \frac{32}{\mu^2} \delta^2 \sigma_*^2  + \frac{16}{\mu \tau} \outers \sigma_*^2 + \frac{16}{\mu} \outers  \delta^2 \sigma_*^2\\
	&\le \left(1 - \frac{\outers\mu}{4}\right)^k\|x^0-\opt\|^2 + \frac{16}{\mu} \left( \frac {2\delta^2} \mu + \frac \outers {\tau} + \outers  \delta^2  \right) \sigma_*^2.
	\end{align*}
\end{proof}

\subsection{Proof of \Cref{th:convengence_of_mamlP_no_stepsize}}

\textbf{\Cref{th:convengence_of_mamlP_no_stepsize}}.
	Consider the iterates of \Cref{alg:mamlP} (with general $\delta$) or \Cref{alg:fo_maml} (for which $\delta=\alpha L$).
Let task losses be $L$--smooth and $\mu$--strongly convex and let objective parameter satisfy $\inners \leq \frac {1}{\sqrt 6 L}$. Choose stepsize $ \beta \leq \frac \tau {4 L}$, where $\tau = |T_k|$ is the batch size. Then we have
	\begin{align*}
		\E{\norm{x^k-x^*}^2} \leq \left(1 - \frac {\beta \mu}{12}  \right)^k \norm{x^0 - x^*}^2 + \frac { 6\left( \frac \beta \tau + 3 \delta^2 \inners^2 L \right) \varopt} {\mu}.
	\end{align*}

\begin{proof}
	Denote $L_\meta$, $\mu_\meta, \kappa_\meta = \frac \Lmeta {\mu_\meta}$ smoothness constant, strong convexity constant, condition number of Meta-Loss functions $\meta_1, \dots, \meta_n$, respectively. 
	We have
	\begin{align*}
	\norm{x^{k+1}-x^*}^2
	&= \biggl\|x^k - x^* - \frac \beta \tau \sum_{i \in T_k}  \nabla \meta_i (y_i^k)\biggr\|^2\\
	&=\norm{x^k - x^*}^2 - \frac {2 \beta} \tau \sum_{i \in T_k} \langle \nabla \meta_i (y_i^k), x^k-x^* \rangle + \beta^2 \norm{ \frac 1 \tau \sum_{i \in T_k} \nabla  \meta_i(y_i^k)}^2\\
	&\leq \norm{x^k - x^*}^2 - \frac {2 \beta} \tau \sum_{i \in T_k} \langle \nabla \meta_i (y_i^k) - \nabla \meta_i(x^*), x^k-x^* \rangle + 2\beta^2 \norm{ \frac 1 \tau \sum_{i \in T_k} (\nabla  \meta_i(y_i^k) - \nabla  \meta_i(x^*) ) }^2\\
	& \quad  - \frac {2 \beta} \tau \sum_{i \in T_k} \langle \nabla \meta_i (x^*), x^k-x^* \rangle + 2 \beta^2 \biggl\| \frac 1 \tau \sum_{i \in T_k} \nabla  \meta_i(x^*)\biggr\|^2.
	\end{align*}
	
	Using Proposition~\ref{pr:three_point}, we rewrite the scalar product as
	$
	\langle \nabla \meta_i (y_i^k) - \nabla \meta_i(x^*), x^k-x^* \rangle  
	= D_{\meta_i}(x^*, y_i^k) + D_{\meta_i}(x^k, x^*) - D_{\meta_i}( x^k, y_i^k) 
	$,
	which gives
	\begin{align*}
	\norm{x^{k+1}-x^*}^2
	& \leq \norm{x^k - x^*}^2 - \frac {2 \beta} \tau \sum_{i \in T_k} \left[ D_{\meta_i}(x^*, y_i^k) + D_{\meta_i}(x^k, x^*) - D_{\meta_i}( x^k, y_i^k) \right]\\
	&+ 2\beta^2 \biggl\| \frac 1 \tau \sum_{i \in T_k} (\nabla  \meta_i(y_i^k) - \nabla  \meta_i(x^*) ) \biggr\|^2  - \frac {2 \beta} \tau \sum_{i \in T_k} \langle \nabla \meta_i (x^*), x^k-x^* \rangle + 2 \beta^2 \biggl\| \frac 1 \tau \sum_{i \in T_k} \nabla  \meta_i(x^*)\biggr\|^2.
	\end{align*}
	Since we sample $T_k$ uniformly and $\{\nabla \meta_i (x^*)\}_{i\in T_k}$ are independent random vectors, we obtain
	\begin{align*}
	\E{\norm{x^{k+1}-x^*}^2}
	&\leq \norm{x^k - x^*}^2 + \frac {2 \beta} \tau \E{ \sum_{i \in T_k} \left[ - D_{\meta_i}(x^*, y_i^k) - D_{\meta_i}(x^k, x^*) + D_{\meta_i}( x^k, y_i^k)\right] }\\
	& \quad + \frac {2\beta^2} {\tau^2} \E{ \sum_{i \in T_k}  \norm{ \nabla  \meta_i(y_i^k) - \nabla  \meta_i(x^*) }^2 } + \frac {2 \beta^2} \tau \varopt.
	\end{align*}
	Next, we are going to use the following three properties of Bregman divergence:
	\begin{align}
	-D_{\meta_i}( x^*, y_i^k)
	& \overset{\eqref{eq:smooth_conv}}{\le} -\frac 1 {2 L_\meta} \norm{ \nabla  \meta_i(y_i^k) - \nabla  \meta_i(x^*) }^2  \notag \\
	-D_{\meta_i}(x^k, x^*)
	&\le  -\frac {\mu_\meta} 2\norm {x^k-x^*}^2\label{eq:div2} \\
	D_{\meta_i}(x^k, y_i^k)
	& \le \frac \Lmeta 2 \norm{x^k-y_i^k}^2. \notag 
	\end{align}
	Moreover, using identity $y_i^k = z_i^k + \alpha \nabla F_i(y_i^k)$, we can bound the last divergence as
	\begin{align*}
	D_{\meta_i}(x^k, y_i^k)
	&\le \frac \Lmeta 2 \norm{x^k-z_i^k - \inners \nabla \meta_i (y_i^k)}^2\\
	&= \frac 1 2 \inners^2 \Lmeta \Bigl\|\frac 1 \inners (x^k-z_i^k) - \nabla \meta_i(y_i^k)\Bigr\|^2\\
	& \le \frac 3 2 \inners^2 \Lmeta \left( \Bigl\|\frac 1 \inners (x^k-z_i^k) - \nabla \meta_i(x^k)\Bigr\|^2 + \norm{\nabla \meta_i (x^k)- \nabla \meta_i(x^*)}^2 + \norm{\nabla \meta_i (x^*)- \nabla \meta_i(y_i^k)}^2  \right)\\
	& \le \frac 3 2 \inners^2 \Lmeta \left( \delta^2 \norm{\nabla \meta_i(x^k)}^2 + \norm{\nabla \meta_i (x^k)- \nabla \meta_i(x^*)}^2 + \norm{\nabla \meta_i (x^*)- \nabla \meta_i(y_i^k)}^2  \right),
	\end{align*}
	where the last step used the condition in \Cref{alg:mamlP}. Using inequality~\eqref{eq:a_plus_b} on $\nabla \meta_i(x^k) = \nabla \meta_i(x^*) + (\nabla \meta_i(x^k) - \nabla \meta_i(x^*))$, we further derive
	\begin{align*}
	D_{\meta_i}(x^k, y_i^k)
	& \le \frac 3 2 \inners^2 \Lmeta \left( 2\delta^2 \norm{\nabla \meta_i(x^*)}^2 + (1+2\delta^2) \norm{\nabla \meta_i (x^k)- \nabla \meta_i(x^*)}^2 + \norm{\nabla \meta_i (x^*)- \nabla \meta_i(y_i^k)}^2  \right)\\
	& \overset{\eqref{eq:smooth_conv}}{\le} \frac 3 2 \inners^2 \Lmeta \left( 2\delta^2 \norm{\nabla \meta_i(x^*)}^2 + (1+2\delta^2)\Lmeta D_{\meta_i}(x^k,x^*) + \norm{\nabla \meta_i (x^*)- \nabla \meta_i(y_i^k)}^2  \right).	\end{align*}
	Assuming $\inners\leq \sqrt{\frac 2 3 (1+2\delta^2)} \frac 1 {L_\meta}$ so that $ 1- \frac 3 2 \inners^2 \Lmeta^2 (1+2\delta^2) >0$, we get
	\begin{align*}
	-D_{\meta_i}(x^k, x^*) + D_{\meta_i}(x^k, y_i^k)
	&\le - \left( 1- \frac 3 2 \inners^2 \Lmeta^2 (1+2\delta^2) \right) D_{\meta_i}(x^k,x^*) \\
	&\quad + \frac 3 2 \inners^2 \Lmeta \left( 2\delta^2 \norm{\nabla \meta_i(x^*)}^2 + \norm{\nabla \meta_i (x^*)- \nabla \meta_i(y_i^k)}^2  \right)\\
	&\stackrel{\eqref{eq:div2}}\le - \left( 1- \frac 3 2 \inners^2 \Lmeta^2 (1+2\delta^2) \right) \frac {\mu_\meta} 2\norm {x^k-x^*}^2\\
	&\quad +  \frac 3 2 \inners^2 \Lmeta \left( 2\delta^2 \norm{\nabla \meta_i(x^*)}^2 + \norm{\nabla \meta_i (x^*)- \nabla \meta_i(y_i^k)}^2  \right).
	\end{align*}
	
	Plugging these inequalities yields
	\begin{align*}
	\E{\norm{x^{k+1}-x^*}^2}	
	&\leq \left(1 - \beta \mu_\meta \left( 1- \frac 3 2 \inners^2 \Lmeta^2 (1+2\delta^2) \right)  \right) \norm{x^k - x^*}^2 \\
	& \quad+  \frac \beta \tau \left(3 \inners^2 \Lmeta +  \frac {2\beta} \tau - \frac 1 {L_\meta}  \right) \E{ \sum_{i \in T_k}  \norm{ \nabla  \meta_i(y_i^k) - \nabla  \meta_i(x^*) }^2 }\\
	& \quad + 2 \beta\left( \frac \beta \tau + 3 \inners^2 \delta^2 \Lmeta \right)\varopt.
	\end{align*}

	Now, if $\inners \leq \frac {1}{\sqrt 6 L_\meta}$ and $ \beta \leq \frac \tau {4 L_\meta}$, then $3 \inners^2 \Lmeta +  \frac {2\beta} \tau  - \frac 1 {L_\meta} \leq 0$, and consequently
	\begin{align*}
	\E{\norm{x^{k+1}-x^*}^2}
	&\leq \left(1 - \beta \mu_\meta \left( 1- \frac 3 2 \inners^2 \Lmeta^2 (1+2\delta^2) \right)  \right) \norm{x^k - x^*}^2 + 2 \beta \left( \frac \beta \tau + 3 \inners^2 \delta^2 \Lmeta \right)\varopt.
	\end{align*}
	
	We can unroll the recurrence to obtain the rate
	\begin{align*}
	\E{\norm{x^k-x^*}^2}
	& \le \left(1 - \beta \mu_\meta \left( 1- \frac 3 2 \inners^2 \Lmeta^2 (1+2\delta^2) \right)  \right)^k \norm{x^0 - x^*}^2 \\
	&\quad + \left( \sum_{i=0}^{k-1} \left(1 - \beta \mu_\meta \left( 1- \frac 3 2 \inners^2 \Lmeta^2 (1+2\delta^2) \right)  \right)^i \right) 2 \beta \left( \frac \beta \tau + 3 \inners^2 \delta^2 \Lmeta \right)\varopt\\
	&= \left(1 - \beta \mu_\meta \left( 1- \frac 3 2 \inners^2 \Lmeta^2 (1+2\delta^2) \right)  \right)^k \norm{x^0 - x^*}^2 \\
	&\quad + \left( \frac {1- \left(1 - \beta \mu_\meta \left( 1- \frac 3 2 \inners^2 \Lmeta^2 (1+2\delta^2) \right)  \right)^k }{1- \frac 3 2 \inners^2 \Lmeta^2 (1+2\delta^2)} \right) \frac 2 {\mu_\meta} \left( \frac \beta \tau + 3 \inners^2 \delta^2 \Lmeta \right)\varopt \\
	&\leq \left(1 - \beta \mu_\meta \left( 1- \frac 3 2 \inners^2 \Lmeta^2 (1+2\delta^2) \right)  \right)^k \norm{x^0 - x^*}^2 + \frac {2 \left( \frac \beta \tau + 3 \inners^2 \delta^2 \Lmeta \right) \varopt} {\mu_\meta (1- \frac 3 2 \inners^2 \Lmeta^2 (1+2\delta^2))}.
	\end{align*}
	
	Choice of $\delta$ implies $0 \leq \delta \leq 1$; \Cref{pr:moreau_is_smooth} yields $\frac \mu 2 \leq \frac \mu {1+\inners \mu} \leq \mu_\meta$ and $L_\meta \leq \frac L {1 + \inners L} \leq L$, so we can simplify
	\begin{align*}
	\E{\norm{x^k-x^*}^2} \leq \left(1 - \frac {\beta \mu}2 \left( 1- 5 \inners^2 L^2 \right)  \right)^k \norm{x^0 - x^*}^2 + \frac { 4\left( \frac \beta \tau + 3 \inners^2 L \delta^2 \right) \varopt} {\mu (1 - 2 \inners^2 L^2 )}.
	\end{align*}
\end{proof}
\subsection{Proof of \Cref{th:nonconvex_fo_maml}}
\textbf{Theorem~\ref{th:nonconvex_fo_maml}} 
	Let \Cref{as:bounded_var} hold, functions $f_1,\dotsc, f_n$ be $L$--smooth and $F$ be lower bounded by $F^*>-\infty$. Assume $\alpha\le \frac{1}{4L}, \beta\le \frac{1}{16L}$. If we consider the iterates of \Cref{alg:fo_maml} (with $\delta=\alpha L$) or \Cref{alg:mamlP} (with general $\delta$), then
	\begin{align*}
		\min_{t\le k}\E{\|\nabla F(x^t)\|^2}
		\le \frac{4}{\beta k}\E{F(x^0)-F^*}+ 16 \beta(\alpha L)^2 \left(\frac{1}{|T_k|} + (\alpha L)^2\delta^2\right) \sigma^2.
	\end{align*}
\begin{proof}
	We would like to remind the reader that for our choice of $z_i^k$ and $y_i^k$, the following three identities hold. Firstly, by definition $y_i^k = z_i^k + \alpha \nabla f_i(z_i^k)$. Secondly, as shown in Lemma~\ref{lem:explicit_grad_to_implicit}, $ z_i^k = y_i^k - \alpha\nabla \meta_i(y_i^k)$. And finally, $\nabla f_i(z_i^k) = \nabla \meta_i(y_i^k)$. We will frequently  use these identities in the proof.
	
	Since functions $f_1,\dotsc, f_n$ are $L$-smooth and $\alpha \le \frac{1}{4L}$, functions $F_1,\dotsc, F_n$ are $(2L)$-smooth as per \Cref{lem:moreau_is_str_cvx_and_smooth}. Therefore, by smoothness of $F$, we have for the iterates of \Cref{alg:mamlP}
	\begin{align*}
		\E{F(x^{k+1})}
		&\overset{\eqref{eq:smooth_func_vals}}{\le} \E{F(x^k) + \<\nabla F(x^k), x^{k+1}-x^k> + L\|x^{k+1}-x^k\|^2 }\\
		&= \E{F(x^k) - \beta  \biggl\langle\nabla F(x^k), \aveis i {T_k} \nabla f_i(z_i^k) \biggr\rangle + \beta^2 L \biggl\|\aveis i {T_k} \nabla f_i(z_i^k)}\biggr\|^2\\
		&= F(x^k) - \beta \|\nabla F(x^k)\|^2 + \beta\mathbb{E} \Biggl[\biggl\langle\nabla F(x^k), \nabla F(x^k) - \avein \nabla f_i(z_i^k)\biggr\rangle\Biggr]  \\
			& \qquad + \beta^2 L \E{\biggl\| \aveis i{T_k} \nabla f_i(z_i^k)\biggr\|^2 } \\
		&\overset{\eqref{eq:a_plus_b}}{\le} F(x^k) - \frac{\beta}{2} \|\nabla F(x^k)\|^2 + \frac{\beta}{2} \avein \norms{\nabla F_i(x^k) - \nabla f_i(z_i^k)}  + \beta^2 L \E{\biggl\|\aveis i {T_k}\nabla f_i(z_i^k)\biggr\|^2}.
	\end{align*}
	Next, let us observe, similarly to the proof of Lemma~\ref{lem:maml_approx_grad}, that the gradient approximation error satisfies
	\begin{align*}
		\norm{\nabla F_i(x^k) - \nabla f_i(z_i^k)}
		&= \norm{\nabla F_i(x^k) - \nabla \meta_i(y_i^k)}
		\le L\norm{x^k - y_i^k} 
		= L \norm{x^k - z_i^k - \alpha \nabla\meta_i(y_i^k)} \\
		&\le \alpha L\norm{ \nabla \meta(x^k) -\nabla\meta_i(y_i^k)} + \alpha L\Bigl\|\frac{1}{\alpha}(x^k-z_i^k) - \nabla \meta_i(x^k)\Bigr\|\\
		&= \alpha L\norm{\nabla \meta(x^k) -\nabla f_i(z_i^k) } + \alpha L\Bigl\|\frac{1}{\alpha}(x^k-z_i^k) - \nabla \meta_i(x^k)\Bigr\|.
	\end{align*}
	By rearranging and using our assumption on error $\delta$ as formulated in \Cref{alg:mamlP}, we have
	\[
		\norm{\nabla \meta_i(x^k) -\nabla f_i(z_i^k)}
		\le \frac{\alpha L}{1-\alpha L}\Bigl\|\frac{1}{\alpha}(x^k-z_i^k) - \nabla \meta_i(x^k)\Bigr\|
		\le \frac{\alpha L}{1-\alpha L}\delta\|\nabla \meta_i(x^k)\|
		\overset{\alpha\le \frac{1}{4L}}{\le} \frac{4}{3}\alpha L\delta \|\nabla \meta_i(x^k)\|.
	\]
	Squaring this bound and averaging over $i$, we obtain
	\begin{align*}
		\avein\norms{\nabla \meta_i(x^k) -\nabla f_i(z_i^k)}
		&\le \frac{16}{9}(\alpha L)^2\delta^2 \avein\|\nabla \meta_i(x^k)\|^2 \\
		&= \frac{16}{9}(\alpha L)^2\delta^2 \|\nabla \meta(x^k)\|^2 + \frac{16}{9}(\alpha L)^2\delta^2 \avein\|\nabla \meta_i(x^k) - \nabla \meta (x^k)\|^2 \\
		&\overset{\eqref{eq:bounded_var}}{\le} \frac{16}{9}(\alpha L)^2\delta^2 \|\nabla \meta(x^k)\|^2 + \frac{16}{9}(\alpha L)^2\delta^2 \sigma^2 \\
		&\le \frac{1}{9}\|\nabla \meta(x^k)\|^2 + 2 (\alpha L)^2\delta^2 \sigma^2.
	\end{align*}
	For the other term in the smoothness upper bound, we can write
	\begin{align*}
	\E{\biggl\|\frac{1}{|T_k|}\sum_{i\in T_k} \nabla f_i(z_i^k)\biggr\|^2}
	&= \E{\biggl\|\frac{1}{|T_k|}\sum_{i\in T_k} \nabla\meta_i(x^k) + \frac{1}{|T_k|}\sum_{i\in T_k} (\nabla f_i(z_i^k) - \nabla\meta_i(x^k))\biggr\|^2} \\
	&\overset{\eqref{eq:a_plus_b}}{\le} 2\E{\biggl\|\frac{1}{|T_k|}\sum_{i\in T_k} \nabla\meta_i(x^k)\biggr\|^2} + 2\E{\biggl\| \frac{1}{|T_k|}\sum_{i\in T_k} (\nabla f_i(z_i^k) - \nabla\meta_i(x^k))\biggr\|^2} \\
	&\overset{\eqref{eq:jensen_for_sq_norms}}{\le} 2\E{\biggl\|\frac{1}{|T_k|}\sum_{i\in T_k} \nabla\meta_i(x^k)\biggr\|^2} + \frac{2}{|T_k|}\E{\sum_{i\in T_k}\| \nabla f_i(z_i^k) - \nabla\meta_i(x^k)\|^2} \\
	&\le 2\E{\biggl\|\frac{1}{|T_k|}\sum_{i\in T_k} \nabla\meta_i(x^k)\biggr\|^2} + \E{\frac{32}{9}\frac{1}{|T_k|}\sum_{i\in T_k} (\alpha L)^2\delta^2 \| \nabla\meta_i(x^k)\|^2}.
	\end{align*}
	Using bias-variance decomposition, we get for the first term in the right-hand side
	\begin{align*}
		2\E{\biggl\|\frac{1}{|T_k|}\sum_{i\in T_k} \nabla\meta_i(x^k)\biggr\|^2}
		&\overset{\eqref{eq:rand_vec_sq_norm}}{=} 2\|\nabla \meta (x^k)\|^2 + 2\E{\biggl\|\frac{1}{|T_k|}\sum_{i\in T_k} \nabla\meta_i(x^k) - \nabla\meta(x^k)\biggr\|^2} \\
		&= 2\|\nabla \meta (x^k)\|^2 + \frac{2}{|T_k|}\frac{1}{n}\sum_{i=1}^n\| \nabla\meta_i(x^k) - \nabla\meta(x^k)\|^2.
	\end{align*}
	Similarly,  we simplify the second term using $\frac{32}{9}< 4$ and then obtain
	\begin{align*}
		\frac{32}{9}\E{\frac{1}{|T_k|}\sum_{i\in T_k} (\alpha L)^2\delta^2 \| \nabla\meta_i(x^k)\|^2}
		\overset{\eqref{eq:rand_vec_sq_norm}}{\le} 4(\alpha L)^2\delta^2\|\nabla\meta(x^k)\|^2 + \frac{4(\alpha L)^2\delta^2}{n}\sum_{i=1}^n\| \nabla\meta_i(x^k) - \nabla\meta(x^k)\|^2.
	\end{align*}
	Thus, using $\alpha\le\frac{1}{4L}$ and $\delta\le 1$, we get
	\begin{align*}
	\E{\biggl\|\frac{1}{|T_k|}\sum_{i\in T_k} \nabla f_i(z_i^k)\biggr\|^2}
	&\le 3\|\nabla \meta (x^k)\|^2 + \left(\frac{2}{|T_k|} + 4(\alpha L)^2\delta^2\right)\sum_{i=1}^n\| \nabla\meta_i(x^k) - \nabla\meta(x^k)\|^2 \\
	&\overset{\eqref{eq:bounded_var}}{\le} 3\|\nabla \meta (x^k)\|^2 + 4\left(\frac{1}{|T_k|} + (\alpha L)^2\delta^2\right)\sigma^2.
	\end{align*}
	Now we plug these inequalities back and continue:
	\begin{align*}
		\E{F(x^{k+1})} - F(x^k)
		& \leq - \frac{\beta}{2} \|\nabla F(x^k)\|^2 + \frac{\beta}{18}\|\nabla\meta (x^k)\|^2 + \beta(\alpha L)^2\delta^2 \sigma^2 \\
		&\quad  + 3\beta^2 L\|\nabla \meta (x^k)\|^2 + 4\beta^2 L\sigma^2\left(\frac{1}{|T_k|} + (\alpha L)^2\delta^2\right)\sigma^2 \\
		&\overset{\beta\le \frac{1}{16L}}{\le }   - \frac{\beta}{4} \|\nabla F(x^k)\|^2 + 4\beta^2 L\sigma^2\left(\frac{1}{|T_k|} + (\alpha L)^2\delta^2\right)\sigma^2 + \beta(\alpha L)^2\delta^2 \sigma^2.
	\end{align*}
	Rearranging the terms and telescoping this bound, we derive
	\begin{align*}
		\min_{t\le k}\E{\|\nabla F(x^t)\|^2}
		&\le \frac{4}{\beta k}\E{F(x^0)-F(x^{k+1})}+ 16 \beta \left(\frac{1}{|T_k|} + (\alpha L)^2\delta^2\right) \sigma^2 + 4(\alpha L)^2\delta^2 \sigma^2\\
		&\le \frac{4}{\beta k}\E{F(x^0)-F^*}+ 16 \beta \left(\frac{1}{|T_k|} + (\alpha L)^2\delta^2\right) \sigma^2 + 4(\alpha L)^2\delta^2 \sigma^2.
	\end{align*}
	The result for \Cref{alg:fo_maml} is obtained as a special case with $\delta=\alpha L$.
\end{proof}

\end{document}